\documentclass[a4paper,12pt,reqno]{amsart}

\usepackage{amssymb}
\usepackage{amsmath}
\usepackage{amsthm}

\usepackage{rcs}

\RCS $Id: HDL-conserv-bc-elsevier.tex,v 1.16 2015/05/07 14:11:36 nisikawa Exp $

\DeclareFontFamily{U}{rsfs}{\skewchar\font"7F}
\DeclareFontShape{U}{rsfs}{m}{n}{
	<-6> rsfs5
	<6-8> rsfs7
	<8-> rsfs10
	}{}
\DeclareMathAlphabet{\mathscr}{U}{rsfs}{m}{n}

\theoremstyle{plain}
\newtheorem{theorem}{Theorem}[section]
\newtheorem{prop}[theorem]{Proposition}
\newtheorem{lemma}[theorem]{Lemma}
\newtheorem{cor}[theorem]{Corollary}

\theoremstyle{definition}
\newtheorem{define}[theorem]{Definition}
\theoremstyle{remark}
\newtheorem{ass}{Assumption}[section]

\newtheorem{remark}{Remark}[section]

\def\real{{\mathbb R}}
\def\natural{{\mathbb N}}
\def\integer{{\mathbb Z}}

\def\dive{\mathop{\mathrm{div}}\nolimits}

\def\supp{\mathop{\mathrm{supp}}\nolimits}

\numberwithin{equation}{section}

\begin{document}
\title[Hydrodynamic limit for the
        interface model with a conservation law]{Hydrodynamic limit
	for the Ginzburg-Landau $\nabla\phi$
	interface model with a conservation law and Dirichlet
	boundary conditions}
\author{Takao Nishikawa}
\address{Department of Mathematics, College of Science and Technology,
Nihon University, \\ 1-8-14 Kanda-Surugadai, Chiyoda-ku, Tokyo 101-8308, Japan}
\email{nisikawa@math.cst.nihon-u.ac.jp}
\begin{abstract}
	Hydrodynamic limit for the Ginzburg-Landau $\nabla\phi$ interface model
	with a conservation law was established in \cite{N02a} under periodic boundary conditions. This paper studies the same problem on a
	bounded domain imposing the Dirichlet boundary condition.
	A nonlinear partial equation of fourth order satisfying
	the boundary conditions is derived as the macroscopic equation.
	Its solution converges to the Wulff shape
	derived by \cite{DGI00} as the time $t\to\infty$.
\end{abstract}

\maketitle
%
%
\section{Introduction}
The Ginzburg-Landau $\nabla\phi$ interface model determines stochastic dynamics for a discretized hypersurface separating two microscopic phases embedded in the $d+1$ dimensional space. The position of the hypersurface is described by height variables $\phi=\{\phi(x);x\in\Gamma\}$ measured from a fixed $d$-dimensional
discrete hyperplane $\Gamma$. We will take $\Gamma=\Gamma_N:=(\integer/N\integer)^d$ when we consider
the system on a discretized torus with periodic boundary condition, or
$\Gamma=D_N\subset \integer^d$ when we consider the system on a domain
with some boundary condition. Here, 
$D_N$ is a microscopic domain corresponding to a given macroscopic
domain $D\subset\real^d$ which is bounded and has a smooth boundary;
see Section~\ref{sec-model} for the definition of $D_N$.
We then admit an energy (Hamiltonian) for the interface $\phi$ by
\[
	H(\phi)=\frac{1}{2}\sum_{\begin{subarray}{c}
		x,y\in\Gamma, \\ |x-y|=1
	\end{subarray}}V(\phi(x)-\phi(y))
	+\sum_{\begin{subarray}{c}
		x\in\Gamma,y\in\integer^d\smallsetminus\Gamma, \\ |x-y|=1
	\end{subarray}}V(\phi(x)-\phi(y))
\]
with a potential $V\in C^2(\real)$. Note that, when $\Gamma=D_N$,
we need to give 
a boundary condition $\{\phi(x); x\in \integer^d\smallsetminus D_N\}$
in order to define the Hamiltonian $H$.
Once we introduce the Hamiltonian $H$, the dynamics of the interface
can be introduced by means of the Langevin equation
\begin{equation}\label{SDE0}
	d\phi_t(x)=-\frac{\partial H}{\partial \phi(x)}(\phi_t)\,dt
	+\sqrt{2}dw_t(x),\quad x\in \Gamma,
\end{equation}
where $\{w_t(x);x\in\Gamma\}$ is a family of independent copies of
the one dimensional standard Brownian motion.
The hydrodynamic scaling limit has been established for 
the dynamics governed by \eqref{SDE0} with
$\Gamma=\Gamma_N$ in \cite{FS97}, for
the dynamics with $\Gamma=D_N$ and the Dirichlet boundary condition in \cite{N03}. In both cases, the macroscopic motion is described by the nonlinear
partial differential equation
\begin{align}
	\frac{\partial}{\partial t}h(t,\theta)&=
	\dive\Bigl\{(\nabla\sigma)(\nabla h(t,\theta))
	\Bigr\} \nonumber \\
	& \equiv\sum_{i=1}^d
	\frac{\partial}{\partial\theta_i}
	\left\{\frac{\partial\sigma}{\partial u_i}
	\left(\nabla h(t,\theta)\right)\right\}
	,\quad\theta\in D,\,t>0 \nonumber 
\end{align}
with an appropriate boundary condition.
Here, $\sigma:\real^d\to\real$ is a function called ``surface tension,''
which gives the local energy of macroscopic interface with tilt $u\in\real^d$,
see \cite{FS97} or \cite{F05}for precise definition.

The dynamics \eqref{SDE0} can be regarded as the model
corresponding to the Glauber dynamics in the particles' systems.
Let us introduce the model
corresponding to the Kawasaki dynamics in the particles' systems as follows:
\begin{equation}\label{SDE1a}
	d\phi_t(x)=-(-\Delta_\Gamma)
	\frac{\partial H}{\partial \phi(\cdot)}(\phi_t)(x)\,dt
	+\sqrt{2}d\tilde{w}_t(x),\quad x\in \Gamma,
\end{equation}
where $\Delta_\Gamma$ is the discrete Laplacian on $\Gamma$ with the Neumann boundary condition when $\Gamma=D_N$ or periodic boundary condition when
$\Gamma=\Gamma_N$ defined by
\begin{align}
	(\Delta_\Gamma\psi)(x)&\equiv\sum_{y\in\Gamma}\Delta_\Gamma(x,y)\psi(y) 
	\nonumber \\
	&=\sum_{y\in\Gamma,|x-y|=1}\left\{\psi(y)-\psi(x)\right\},
	\quad \psi\in\real^\Gamma,\,x\in\Gamma,\label{laplacian-gamma}
\end{align}
and $\{\tilde{w}_t(x);x\in\Gamma\}$ is a family of Gaussian processes
with mean zero and covariance structure
\[
	E[\tilde{w}_t(x)\tilde{w}_s(y)]=-\Delta_\Gamma(x,y)t\wedge s,\quad
	x,y\in\Gamma, t,s\ge 0.
\]
We note that the dynamics \eqref{SDE1a} preserves the sum
$\sum_{x\in\Gamma}\phi_t(x)$, which can be regarded as the volume of the
phase located below the interface.

The main purposes of this paper are to establish the hydrodynamic scaling limit
of $\phi_t$ determined by \eqref{SDE1a} with $\Gamma=D_N$
under the Dirichlet boundary condition
\[
\phi_t(x)=0,\quad x\in \integer^d\smallsetminus D_N,\,t\ge 0,
\]
and to clarify the relationship between the macroscopic motion and
``Wulff shape'' studied by \cite{DGI00} under a static situation.
The main result in this paper is that, under the scaling $N^4$ for time
and $N$ for space, the macroscopic motion corresponding to $\phi_t$ is described
by the nonlinear partial differential equation with Dirichlet boundary condition
\begin{equation}\label{PDE1} 
\left\{\begin{split}
	\frac{\partial}{\partial t}h(t,\theta)&=
	-\Delta \dive\Bigl\{(\nabla\sigma)(\nabla h(t,\theta))
	\Bigr\} \\
	& \equiv -\sum_{i=1}^d\sum_{j=1}^d\frac{\partial^2}{\partial\theta_j^2}
	\frac{\partial}{\partial\theta_i}
	\left\{\frac{\partial\sigma}{\partial u_i}
	\left(\nabla h(t,\theta)\right)\right\}
	,\quad\theta\in D,\,t>0 \\
	h(t,\theta)&=0,\quad \theta\in D^c,t>0.
\end{split}\right.
\end{equation}

We note that the equation \eqref{PDE1} is the gradient flow
in a certain affine subspace in $H^1(D)^*$
with respect to the energy functional
\[
	\Sigma_D(h)=\int_D \sigma(\nabla h(\theta))\,d\theta
\]
with the Dirichlet boundary condition $h|_{D^\complement}=0$,
where $H^1(D)^*$ is the topological dual of $H^1(D)$,
see Section~\ref{subsec-mainresult} for details.
Here, $\sigma:\real^d\to\real$ is a function called ``surface tension,''
which gives the local energy of macroscopic interface with tilt $u\in\real^d$,
see \cite{FS97} for precise definition.
The functional $\Sigma_D$ is called ``total surface tension,''
which gives the total energy of the interface $h$.
We note that the total surface tension $\Sigma_D$ coincides with 
the rate functional for the large deviation principle under the
static situation, see \cite{DGI00}.
Taking $\Gamma=\Gamma_N$ instead of $D_N$, the large scale hydrodynamic
behavior has been studied by \cite{N02a}.

We should also mention the relationship between the equation \eqref{PDE1}
and the Wulff shape discussed in \cite{DGI00}. As an application of 
the large deviation principle, the macroscopic height variable $h^N$
under the equilibrium state (Gibbs measure) conditioned on the total volume
converges to the macroscopic interface so-called ``Wulff shape,'' which
is characterized as the solution of the variational problem
\begin{equation}\label{variational}
	\arg\inf\left\{\Sigma_D(h);\, h\in H_0^1(D),
	 \int_D h(\theta)\,d\theta=v\right\}
\end{equation}
as $N\to\infty$, where $v$ is the limit of the volume of $h^N$
rescaled by $N^{-d}$.
We emphasize that the solution $h(t)$ of \eqref{PDE1} converges
as $t\to\infty$, and the limit coincides with the unique
solution of the variational
problem \eqref{variational}. Indeed, the macroscopic motion
described by \eqref{PDE1} relaxes the total energy $\Sigma_D$ and
converges to its minimizer, that is, the Wulff shape as $t\to\infty$.

Before closing the introduction, let us briefly give the organization of
this paper.
In Section~\ref{sec-model}, we formulate our problem more precisely
and state the main result. In Section~\ref{sec-PDE}, we study several
properties of the macroscopic equation \eqref{PDE1} and its spatial
discretization. In Section~\ref{sec-stationary}, we show that
a translation-invariant stationary measure for the dynamics $\phi_t$
on the infinite lattice $\integer^d$ need to be a canonical Gibbs measure
corresponding to the Hamiltonian $H$. Combining the above with
the known result in \cite{N02a}, we have the characterization
of the family of translation-invariant stationary measures of the
dynamics on $\integer^d$.
In Section~\ref{sec-pf-mainthm}, we derive the macroscopic equation
\eqref{PDE1} from the stochastic dynamics \eqref{SDE1a} with $\Gamma=D_N$,
after establishing several estimates on it.

\section{Model and main results}\label{sec-model}
\subsection{Model}\label{subsec-model}
Let $D$ be a bounded and connected domain in $\real^d$
with a Lipschitz boundary.
For convenience, we assume that $D$ contains the origin of $\real^d$.
Furthermore, we assume the following condition on $D$:
\begin{ass}\label{ass-domain}
	Let $\tilde{D}_N=ND\cap \integer^d$.  
	We assume that there exists a constant $C>0$ independent of $N$
	such that
	\[d_{\tilde{D}_N}(x,y)\le C,\quad x,y\in \tilde{D}_N, |x-z|\le 2,
	|y-z|\le 2\]	
	for every $N\ge1$ large enough and $z\in \integer^d\smallsetminus \tilde{D}_N$, where $d_{\tilde{D}_N}$ is the ordinary graph distance on $\tilde{D}_N$.
\end{ass}
\begin{remark}
Assumption~\ref{ass-domain} is satisfied if the domain $D$ is convex.
\end{remark}

Let us introduce the discretized microscopic domain corresponding to $D$.
To keep notation simple, we consider $D_N\subset\integer^d$ defined
by 
\[D_N=\{x\in \tilde{D}_N;\,B(x/N,5/N)\subset D\},\]
where $B(\alpha,l)$ stands for the hypercube in $\real^d$
with center $\alpha=(\alpha_i)_{i=1}^d$ and side length $l>0$, that is,
\[B(\alpha,l)=\prod_{i=1}^d[\alpha_i-l/2,\alpha_i+l/2).\]
On $D_N$ we consider the dynamics governed
by the following stochastic differential equations (SDEs)
\begin{equation}\label{SDE1}
	d\phi_t(x)=-(-\Delta_{D_N})U_\cdot(\phi_t)(x)\,dt+\sqrt{2}d\tilde w_t(x),\quad x\in D_N,
\end{equation}
with the Dirichlet boundary condition 
\begin{equation}\label{bc}
	\phi_t(x)=0,\quad x\in\integer^d\smallsetminus D_N
\end{equation}
and an initial data $\phi_0=\{\phi_0(x);\,x\in\integer^d\}$ satisfying
the condition \eqref{bc} with $t=0$,
where $U_x(\phi)$ in the drift term is defined by
\begin{equation}\label{drift-term}
	U_x(\phi):=\frac{\partial H}{\partial\phi(x)}(\phi)
	\equiv \sum_{y\in\integer^d;\,|x-y|=1}V'(\phi(x)-\phi(y))
\end{equation}
for $\phi\in\real^{\integer^d}$ and $x\in D_N$.
Here, $\Delta_{D_N}$ is the Laplacian with the Neumann boundary
condition on $D_N$, that is,
$\Delta_{D_N}$ is defined by
\begin{equation}\label{discrete-Laplacian}
	\Delta_{D_N}\psi(x)=\sum_{y\in D_N;\, |x-y|=1}(\psi(y)-\psi(x))
\end{equation}
for $\psi:\real^{D_N}\to\real$.
The process $\tilde w_t=\{\tilde w_t(x);\,x\in D_N\}$ is a family of Gaussian processes
with mean zero and covariance structure
\[
E[\tilde w_{t}(x)\tilde w_{s}(y)]=(-\Delta_{D_N})(x,y)t\wedge s,\quad x,y\in D_N,
\,t,s\ge 0,
\]
where $\Delta_{D_N}(x,y)$ is the kernel of $\Delta_{D_N}$, that is,
$\Delta_{D_N}(x,y)=\Delta_{D_N}1_{y}(x)$. 
Note that a stochastic processes $\tilde w_t$ satisfying the above
can be constructed by
\[\tilde w_t(x)=\sqrt{-\Delta_{D_N}}w_t(x),\]
where $\sqrt{-\Delta_{D_N}}$ is the square root of
$-\Delta_{D_N}$
and $\{w_t(x);\,x\in D_N\}$
is a family of independent one
dimensional Brownian motions. 
For convenience, we extend $\tilde w_t$ to the process on $\integer^d$
by putting $\tilde w_{t}(x)=0$ when $x\in \integer^d\smallsetminus D_N$.

Through out this paper, we always assume the following conditions on $V$:
\begin{ass}\label{ass-potential}
The function $V:\real\to\real$ satisfies following three conditions:
\begin{enumerate}
	\item $V\in C^2(\real)$.
	\item $V$ is symmetric, that is, $V(\eta)=V(-\eta)$ holds
	for all $\eta\in\real$.
	\item There exist constants $c_{+},c_{-}>0$ such that
	\[c_{-}\le V''(\eta)\le c_{+},\quad \eta\in\real\]
	holds. 
\end{enumerate}
\end{ass}
We regard \eqref{SDE1} as the model describing the motion of microscopic
interfaces and introduce the macroscopic height variable $h^N$ by
scaling $N^4$ for time while $N$ for space:
\begin{equation}\label{def-macro-height}
	h^N(t,\theta)=\sum_{x\in\integer^d}N^{-1}\phi_{N^4t}(x)1_{B(x/N,1/N)}(\theta),\quad \theta\in\real^d,
\end{equation}
where $\phi_t=\{\phi_t(x);\,x\in\integer^d\}$ is the solution
of \eqref{SDE1} with \eqref{bc}.
We emphasize that the suitable scaling is not the diffusive one.

\subsection{Main Result}\label{subsec-mainresult}
The main result in this paper is that the macroscopic height variable
defined by \eqref{def-macro-height} converges to the solution of
\eqref{PDE1} as $N\to\infty$. We shall prepare settings
to state the above precisely.
We introduce the triplet
\[H^1(D)\subset L^2(D)\simeq L^2(D)^*\subset H^1(D)^*,\]
where $H^1(D)^*$ is the topological dual of $H^1(D)$.
These spaces are equipped standard norms denoted by
$\|\cdot\|_{H^1(D)},\|\cdot\|_{L^2(D)}$ and $\|\cdot\|_{H^1(D)^*}$,
respectively. 
We denote the duality relation between $H^1(D)^*$ and $H^1(D)$
in this context by ${}_{H^1(D)^*}\langle\cdot,\cdot\rangle_{H^1(D)}$,
namely, ${}_{H^1(D)^*}\langle\cdot,\cdot\rangle_{H^1(D)}$ denotes
the duality relation satisfying 
\[{}_{H^1(D)^*}\langle f,g\rangle_{H^1(D)}
=(f,g)_{L^2(D)},\quad f,g\in L^2(D),\]
where $(\cdot,\cdot)_{L^2(D)}$ is the standard inner product
of $L^2(D)$. 
To make notations simple, we denote
${}_{H^1(D)^*}\langle \cdot,\cdot\rangle_{H^1(D)}$
by $\langle\cdot,\cdot\rangle$ if no confusion arises. 

Under these settings, our main result is 
the following:
\begin{theorem}\label{hydro}
We assume Assumptions~\ref{ass-domain} and \ref{ass-potential}.
We furthermore assume that the sequence of initial data $\phi_0=\phi_0^N$
for \eqref{SDE1} satisfies
\begin{equation}
	\lim_{N\to\infty}E\|h^N(0)-h_0\|_{H^{1}(D)^*}^2=0
\end{equation}
with some $h_0\in L^2(D)$,
where $h^N(0)$ is the macroscopic height variable corresponding to $\phi_0^N$.
Then, for every $t>0$, $h^N(t)$ converges as $N\to\infty$ to $h(t)$, which is
the unique weak solution of the partial differential equation \eqref{PDE1}
%
with the initial data $h_0$.
More precisely, for every $t>0$,
\begin{equation}\label{eq1.7b}
	\lim_{N\to\infty}E\|h^N(t)-h(t)\|_{H^{1}(D)^*}^2=0
\end{equation}
holds.
\end{theorem}

\section{The macroscopic equation and its discretization}\label{sec-PDE}
In this section, we shall focus our attention on the limit equation
\eqref{PDE1} and its discretized version.
The arguments in this section 
highly depend on the properties of the surface tension
$\sigma$ established in \cite{DGI00} and \cite{FS97}.


\subsection{A subspace of $H^1(D)^*$}\label{sec-sobolev}
As we will see, our consideration is on a subspace of $H^{1}(D)^*$.
Before starting discussions, we introduce it with a suitable norm.

We first note that the solution of \eqref{PDE1} should satisfy
\[
	\int_D h(t,\theta)\,d\theta=
	\int_D h_0(\theta)\,d\theta,\quad t\ge0,
\]
which means that the actual state space for \eqref{PDE1} is not
the whole of $H^1(D)^*$. For convenience, let us mainly consider
the time evolution on the tangential space. 
%
Indeed, the equation \eqref{PDE1} will be solved at
\[
H:=\{h\in H^1(D)^*;\,\langle h,1\rangle=0\}.
\]

	On $H$ defined above, we shall introduce another norm which 
	is equivalent to the standard norm $\|\cdot\|_{H^1(D)^*}$
	similarly to \cite{BE91}.
	To do so, we at first introduce the Green operator $G$ for the Laplacian
	with the Neumann boundary condition.
	For $f\in H$, we denote the unique solution $g\in H^1(D)$
	of the elliptic equation
	\begin{gather}
		(\nabla g, \nabla J)_{L^2(D)^d}
		= \langle f, J\rangle,\quad J\in H^1(D), 
		\label{def-green1}\\
		(g,1)_{L^2(D)}=0
		\label{def-green2}
	\end{gather}
	by $Gf$, where $L^2(D)^d$ is the $d$-fold direct product of $L^2(D)$
	and $(\cdot,\cdot)_{L^2(D)^d}$ is its inner product.
	Here, the existence of $Gf$ follows from 
	the Lax-Milgram theorem and the Poincar\'e inequality
	of the following form: 
	there exists a constant $c_D>0$ such that
	\begin{equation}\label{poincare}
	\|f\|^2_{L^2(D)}\le c_D\|\nabla f\|_{L^2(D)^d}^2,\quad f\in H^1(D),\,
	(f,1)_{L^2(D)}=0,
	\end{equation}
	see Section~5.8.1 of \cite{E98} for details.
	We remark that
	the operator $G:H\to H^{1,\mathrm{zm}}(D)$ is bounded and bijective,
	where $H^{1,\mathrm{zm}}(D)$ is defined by
	\[H^{1,\mathrm{zm}}(D):=\{h\in H^1(D);\,(h,1)_{L^2(D)}=0\}.\]
	We also remark that the inverse of $G$ coincides
	with the Laplacian $-\Delta$ restricted to $H^{1,\mathrm{zm}}(D)$.
	
	Let us define the bilinear form $(\cdot,\cdot)_H$ on $H$ by
	\[(h_1, h_2)_H=(\nabla Gh_1,\nabla Gh_2)_{L^2(D)^d},\quad h_1,h_2\in H.\]
	The following proposition tells us 
	that $\|f\|_H:=(f, f)_H^{1/2}$ on $H$ is equivalent
	to $\|f\|_{H^1(D)^*}$ restricted to the closed subspace $H$ 
	of $H^1(D)^*$, and therefore 
	$H$ is a Hilbert space with the inner product $(\cdot,\cdot)_H$.
	\begin{prop}
		There exist two constant $C_1,C_2>0$ such that
		\[C_1\|f\|_{H^1(D)^*}\le\|f\|_H\le
		C_2\|f\|_{H^1(D)^*},\quad f\in H.\] 
	\end{prop}
	\begin{proof}
		
		We shall at first find a constant $C_1>0$ such that
		\begin{equation}\label{ineq-equiv1}
			C_1\|f\|_{H^1(D)^*}\le\|f\|_H,\quad f\in H.
		\end{equation}
		By the definition of the Green operator $G$, we obtain
		\begin{align*}
			\left|\langle f,J\rangle\right|
			&=\left|(\nabla Gf,\nabla J)_{L^2(D)^d}\right| \\
			&\le \|\nabla Gf\|_{L^2(D)^d}\|\nabla J\|_{L^2(D)^d} \\
			&\le (f, f)_H^{1/2}\|J\|_{H^1(D)}
		\end{align*}
		for every $J\in H^1(D)$, by using the Schwarz inequality.
		This implies that \eqref{ineq-equiv1} with $C_1=1$ holds.
		
		We shall next show 
		\begin{equation}\label{ineq-equiv2}
			\|f\|_H\le C_2\|f\|_{H^1(D)^*},\quad f\in H
		\end{equation}
		with some constant $C_2>0$.
		Since \eqref{ineq-equiv2} trivially holds with any
		constant $C_2>0$ when $\|f\|_H=0$,
		we assume $\|f\|_H>0$. 
		Since $\|Gf\|_{H^1(D)}\ge \|f\|_H$ by the definition of
		$\|\cdot\|_H$, we have $\|Gf\|_{H^1(D)}>0$.
		Taking $J=Gf/\|Gf\|_{H^1(D)}$, we obtain
		\begin{align*}
			\langle f,J\rangle
			&=\frac{\|f\|_H^2}{\|Gf\|_{H^1(D)}}
		\end{align*}
		by \eqref{def-green1}. Since we have
		\[\|Gf\|_{H^1(D)}\le \sqrt{1+c_D}\|\nabla Gf\|_{L^2(D)^d}
		=\sqrt{1+c_D}\|f\|_{H}\]
		by the Poincar\'e inequality \eqref{poincare}, we get
		\begin{align*}
			\langle f,J\rangle
			&\ge \left(1+c_D\right)^{-1/2}\|f\|_H.
		\end{align*}		
		By the definition of $\|\cdot\|_{H^1(D)^*}$, we conclude
		\[\|f\|_{H^1(D)^*}\ge \left(1+c_D\right)^{-1/2}\|f\|_H,\]
		which implies \eqref{ineq-equiv2} with $C_2=(1+c_D)^{1/2}$.
	\end{proof}


\subsection{Precise formulation for the macroscopic equation}
We give the precise meaning of the solution of \eqref{PDE1}.
%
Let us introduce a triplet $V\subset H\simeq H^*\subset V^*$
with $H$ introduced in Section~\ref{sec-sobolev},
\[
V=\left\{h\in H_0^1(D);\, \int_D h(\theta)\,d\theta=0\right\}
\]
and the topological dual $V^*$. 
We remark that $V$ is a Hilbert space with the restriction of
the inner product of $H_0^1(D)$ to $V$.
We denote the restriction of the standard norm of $H_0^1(D)$
to $V$ and the dual norm on $V^*$ by
$\|\cdot\|_V$ and $\|\cdot\|_{V^*}$, respectively.
We denote the duality relation between $V^*$ and $V$
by ${}_{V^*}\langle \cdot, \cdot\rangle_{V}$, which satisfies
\[
	{}_{V^*}\langle f, g\rangle_{V}=(f,g)_{H}
\]
for $f\in H$ and $g\in V$. 
Under these settings, we obtain
	\[{}_{V^*}\langle -\Delta f, g\rangle_{V}
	=(-\Delta f, g)_{H}=(f,g)_{L^2(D)}
	={}_{H^1(D)^*}\langle f,g\rangle_{H^1(D)}
	\]
for $f\in H^{1,\mathrm{zm}}(D)$ and $g\in V$,
which indicates that 
$-\Delta$ 
can be extended to a bounded operator from $H$ to $V^*$.
Furthermore, 
we obtain
\begin{equation*} 
	{}_{V^*}\langle -\Delta f, g\rangle_{V}
	=-(u,\nabla g)_{L^2(D)^d}
\end{equation*}
for $g\in H_0^1(D)$ and $f\in C^\infty(D)\cap H$
such that $f=\dive u-(\dive u,1)_{L^2(D)}$ with $u\in C^\infty(D)^d$,
where $C^\infty(D)^d$ is the $d$-fold direct product of $C^\infty(D)$.
Since we have $-\Delta f=-\Delta \dive u$, we finally obtain
\begin{equation}\label{eq2.31}
	{}_{V^*}\langle -\Delta \dive u, g\rangle_{V}
	=-(u,\nabla g)_{L^2(D)^d}
\end{equation}
for $g\in H_0^1(D)$ and $u\in C^\infty(D)^d$.
We can therefore regard the map 
$u\mapsto -\Delta \dive u$ as a bounded operator
from $L^2(D)^d$ to $V^*$
and the identity \eqref{eq2.31} is still valid for $u\in L^2(D)^d$. 
Noting these facts, let us introduce the nonlinear fourth order
differential operator $A_f:V\to V^*$ by
\[
	A_f(h):=-\Delta \dive \left((\nabla\sigma)(\nabla h+\nabla f)\right),\quad h\in V
\]
for a given $f\in H_0^1(D)$. 
We recall that the surface tension $\sigma:\real^d\to\real$ is
$C^1$-class and it satisfies
\begin{equation}\label{eq3.1a}
c_{-}|u-v|^2\le
\left(\nabla\sigma(u)-\nabla\sigma(v)\right)\cdot(u-v)
\le c_{+}|u-v|^2,\quad u,v\in\real^d
\end{equation}
with $c_\pm$ appearing in Assumption~\ref{ass-potential}; see \cite{F05}.

\begin{define}\label{def-PDE}
A function $h=h(t,\theta)$ is called 
the solution of \eqref{PDE1} with initial 
data $h_0\in H^{1}(D)^*$ if there exists a function $f\in H_0^1(D)$ with
\begin{equation}\label{cond-aux-func}
	\langle h_0,1\rangle=
	\langle f,1\rangle
\end{equation}
such that the function $h_f:=h-f$ satisfies the following three
conditions:
\begin{enumerate}
	\item $h_f:[0,T]\to V^*$ is absolutely continuous and
	\[h_f\in C([0,T],H)\cap L^2([0,T],V)\cap W^{1,2}([0,T],V^*).\]
	\item $h_f(0)=h_0-f.$
	\item 
	$h_f$ satisfies
	\begin{equation*}
		\frac{d}{dt}h_f=A_f(h_f(t))
	\end{equation*}
	in $V^*$ for almost every $t\in[0,T]$.
\end{enumerate}
\end{define}
\begin{remark}
\begin{enumerate}
	\item We adopt the weak formulation for \eqref{PDE1}.
	For the solution $h$ of \eqref{PDE1},
	the choice of $V$ 
	corresponds to the Dirichlet boundary condition of $h$.
	Furthermore, $\dive \left((\nabla\sigma)(\nabla h)\right)$
	is formally given by $G\partial h/\partial t$, and therefore
	$\dive \left((\nabla\sigma)(\nabla h+\nabla f)\right)$ should
	satisfy the Neumann boundary condition.
	\item The third condition is equivalent to
	\[h_f(t)=h_f(0)+\int_0^t A_f(h_f(s))\,ds\]
	in $V^*$ for almost every $t\in[0,T]$. 
	The above indicates that a solution $h$ itself satisfies
	\[h(t)-f=h(0)-f+\int_0^t A(h(s))\,ds\]
	in $V^*$. Here, $A$ is the nonlinear fourth order differential
	operator defined by
	\[
		A(h):=-\Delta \dive \left((\nabla\sigma)(\nabla h)\right),
	\]
	which is appearing in the right hand side of \eqref{PDE1}.
\end{enumerate}
\end{remark}

The first aim of this section is to show the existence and uniqueness
for our equation.
\begin{theorem}\label{thm-exist-uniq-PDE}
	For every initial data $h_0\in H^{1}(D)^*$, there exists a unique solution of 
	\eqref{PDE1}.
\end{theorem}
In order to apply a general theory of nonlinear partial differential
equations, let us prepare the following lemma:
\begin{lemma}\label{lem2.13}
	For every
	$f\in H_0^1(D)$,
	the nonlinear operator $A_f:V\to V^*$ satisfies the following:
	\begin{enumerate}
		\item $A_f$ is monotone, that is, 
		\begin{equation*}\label{monotonicity}
		{}_{V^*}\langle A_f(h_1)-A_f(h_2), h_1-h_2\rangle_{V}\le 0,\quad h_1,h_2\in V.
		\end{equation*}
		holds.
		\item $A_f$ is demicontinuous, that is, the map $A_f:V\to V^*$
		is continuous under the weak topology of $V^*$.
		\item There exist constants $C_1,C_2,C_3>0$ such that
		\begin{gather}
			{}_{V^*}\langle A_f(h), h\rangle_{V}\le -C_1\|h\|_V^2+C_2,\quad h\in V,\label{eq2.29}\\
			\|A_f(h)\|_{V^*}\le C_3(\|h\|_V+1),\quad h\in V \label{eq2.30}
		\end{gather}
		holds.
	\end{enumerate}
\end{lemma}

\begin{proof}
	Using \eqref{eq2.31}, we obtain
	\begin{equation}\label{eq2.32a}
		{}_{V^*}\langle A_f(h), g\rangle_{V}
		=-((\nabla\sigma)(\nabla h+\nabla f), \nabla g)_{L^2(D)^d}
	\end{equation}
	for $h,g\in V$, and also obtain
	\begin{align}\label{eq2.32}
		&{}_{V^*}\langle A_f(h_1)-A_f(h_2), g\rangle_{V} \nonumber \\
		&\qquad {}=-((\nabla\sigma)(\nabla h_1+\nabla f)
		-(\nabla\sigma)(\nabla h_2+\nabla f),\nabla g)_{L^2(D)^d}
	\end{align}
	for $h_1,h_2,g\in V$. 
	It is easy to see (1), by applying \eqref{eq2.32} with $g=h_1-h_2$
	and using the convexity of $\sigma$, see \eqref{eq3.1a}.
	We can also obtain (2), since we have
	\begin{align*}
		&\left| {}_{V^*}\langle A_f(h_1)-A_f(h_2), g\rangle_{V}\right|
		 \\ 
		&\qquad \le c_{+}\|\nabla h_1-\nabla h_2\|_{L^2(D)^d}
		\|\nabla g\|_{L^2(D)^d},\quad
		h_1,h_2,g\in V
	\end{align*}
	from \eqref{eq3.1a} and \eqref{eq2.32} again.
	Moreover, we obtain \eqref{eq2.30} of (3), because
	the relationship \eqref{eq2.32a} implies 
	\[\|A_f(h)\|_{V^*}\le \|(\nabla\sigma)(\nabla h+\nabla f)\|_{L^2(D)^d}
	\le c_{+}\|h\|_{V}+c_{+}\|\nabla f\|_{L^2(D)^d}\]
	for $h\in V$.
	
	As the final step, we shall show \eqref{eq2.29} of (3).
	Using \eqref{eq2.32a}, we obtain
	\begin{align*}
		{}_{V^*}\langle A_f(h), h\rangle_{V}
		&=-((\nabla\sigma)(\nabla h+\nabla f), \nabla h+\nabla f)_{L^2(D)^d}\\
		&\qquad {}+((\nabla\sigma)(\nabla h+\nabla f), \nabla f)_{L^2(D)^d} \\
		&\le -c_{-}\|\nabla h+\nabla f\|^2_{L^2(D)^d}
		+c_{+}\|\nabla h+\nabla f\|_{L^2(D)^d}\|\nabla f\|_{L^2(D)^d} \\
		&\le -\frac{1}{2}c_{-}\|\nabla h\|^2_{L^2(D)^d}
		+(c_{+}+c_{-})\|\nabla f\|^2_{L^2(D)^d} \\
		&\qquad {}+c_{+}\|\nabla h\|_{L^2(D)^d}\|\nabla f\|_{L^2(D)^d}
	\end{align*}
	from \eqref{eq3.1a}. Here, noting
	\[\|\nabla h\|_{L^2(D)^d}\|\nabla f\|_{L^2(D)^d}
	\le \frac{c_{-}}{4c_{+}}\|\nabla h\|^2_{L^2(D)^d}+
	\frac{c_{+}}{c_{-}}\|\nabla f\|^2_{L^2(D)^d},\]
	we obtain	
	\begin{align*}
		{}_{V^*}\langle A_f(h), h\rangle_{V}
%
		&\le -\frac{1}{4}c_{-}\|\nabla h\|^2_{L^2(D)^d}
		+\left(\frac{c_{+}^2}{c_{-}}+c_++c_{-}\right)\|\nabla f\|^2_{L^2(D)^d}.
	\end{align*}
	Since we have
	\[\|h\|^2_{V}\le (1+c_D)\|\nabla h\|^2_{L^2(D)^d}\]
	from \eqref{poincare}, we conclude the desired bound \eqref{eq2.29}.
\end{proof}
\begin{proof}[Proof of Theorem~\ref{thm-exist-uniq-PDE}]
Using Theorem~4.10 of \cite{B10} and Lemma~\ref{lem2.13},
for every initial data $h_0\in H^1(D)^*$ and every auxiliary
function $f\in H_0^1(D)$ satisfying \eqref{cond-aux-func},
we obtain the existence and uniqueness of $h_f$ satisfying
conditions (1)-(3) in Definition~\ref{def-PDE}.
Especially, this shows the existence of the solution of \eqref{PDE1}
in the sense of Definition~\ref{def-PDE}.

Let us show the uniqueness of of the solution of \eqref{PDE1}
in the sense of Definition~\ref{def-PDE}.
Take a two solution $h^{(1)}$ and $h^{(2)}$
with a common initial data $h_0$, and let $f_1$ and $f_2$
be auxiliary functions associated to $h^{(1)}$ and $h^{(2)}$, respectively.
Noting 
\[
	A_{f_1}(h)=A_{f_2}(h-f_2+f_1),\quad h\in V,
\]
$h^{(2)}-f_1=h^{(2)}-f_2+f_2-f_1$ satisfies
conditions (1)-(3) in Definition~\ref{def-PDE} with $f_1$.
The uniqueness of $h_f$ for given $f$
implies $h^{(2)}-f_1=h^{(1)}-f_1$, which shows $h^{(1)}=h^{(2)}$.
\end{proof}

We can also obtain the following
by a similar argument to the proof of Theorem~\ref{thm-exist-uniq-PDE}.
\begin{prop}\label{prop2.13b}
	If $h(t)$ is the solution of \eqref{PDE1} with
	initial data $h_0$,
	then $h_f=h-f$ satisfies conditions (1)-(3)
	in Definition~\ref{def-PDE} for every $f\in H_0^1(D)$
	with \eqref{cond-aux-func}. 
	In other words,
	we can choose an auxiliary function $f\in H_0^1(D)$ arbitrarily
	as long as it satisfies \eqref{cond-aux-func}.
%
\end{prop}

\subsection{Regularization for the macroscopic equation}
We shall introduce
the regularization of $\sigma$ and the coefficient of 
the corresponding partial differential
equation \eqref{PDE1}, which plays a key role in the proof of Theorem~\ref{hydro}.
Note that such regularization becomes unnecessary once one can show
$C^2$-regularity of $\sigma$, which remains open at present. 

Let $\rho\in C_0^\infty(\real^d)$ be a non-negative and
symmetric function such that its support is in $
\{u\in\real^d;|u|<1\}$ and $\int_{\real^d}\rho(u)\,du=1$.
For $0<\delta\le 1$, we define $\rho_\delta$ by
\[\rho_\delta=\delta^{-d}\rho(\delta^{-1}u),\quad u\in\real^d\]
and the regularized surface tension $\sigma^\delta$ by
the mollification of $\sigma$:
\[\sigma^\delta(u)=\sigma*\rho_\delta(u),\quad u\in\real^d.\]
Note that the regularized surface tension $\sigma^\delta$
again satisfies the bound \eqref{eq3.1a}, that is,
\begin{equation}
	c_{-}|u-v|^2\le (\nabla\sigma^\delta(u)-\nabla\sigma^\delta(v))\cdot
	(u-v)\le c_{+}|u-v|^2,\quad u,v\in\real^d \label{convex-reg}
\end{equation}
holds for every $0<\delta\le 1$. Moreover, since $\nabla\sigma$ is Lipschitz continuous, $\sigma^\delta$
approximates $\sigma$ in the following sense:
\begin{equation}
	\left|\nabla\sigma^\delta(u)-\nabla\sigma(u))\right|\le c_{+}\delta,
	\quad u\in\real^d,\delta>0. \label{approx-reg}
\end{equation}
Using $\sigma^\delta$ defined above, let us consider
the nonlinear fourth order differential equation
\begin{equation}\label{PDE1-reg}
	\left\{\begin{split}
	\frac{\partial}{\partial t}h^\delta(t,\theta)&=
	-\Delta \dive\Bigl\{(\nabla\sigma^\delta)(\nabla h(t,\theta))
	\Bigr\},\quad\theta\in D,\,t>0 \\
	h^\delta(t,\theta)&=0,\quad \theta\in D^c,t>0.
	\end{split}\right.
\end{equation}
The equation \eqref{PDE1-reg} can be formulated in a similar way
to Definition~\ref{def-PDE}, and the solution of \eqref{PDE1-reg}
exists uniquely. Note that the proof is 
quite parallel with that of Theorem~\ref{thm-exist-uniq-PDE}
since we have \eqref{convex-reg}.
Furthermore, we get the following proposition, which tells us
that the solution of \eqref{PDE1} 
can be approximated by the solution of \eqref{PDE1-reg}.
\begin{prop}\label{prop-PDE-approx}
	Let $h$ and $h^\delta$ be the solutions of \eqref{PDE1} and
	\eqref{PDE1-reg}, respectively. If their initial data
	are common, we have
	\[\lim_{\delta\to 0}\|h(t)-h^\delta(t)\|_{H^1(D)^*}=0.\]
\end{prop}
\begin{proof}
	Let $h_0\in H^1(D)^*$ be the initial data of $h$ and $h^\delta$.
	We take $f\in H_0^1(D)$ satisfying
	$\langle f,1\rangle=\langle h_0,1\rangle$.
	By Definition~\ref{def-PDE} and Proposition~\ref{prop2.13b}, we obtain
	{\allowdisplaybreaks
	\begin{align*}
		\|h^\delta(t)-f\|_{H}^2
		&=\|h_0-f\|_{H}^2 \\
		&\qquad {}-2\int_0^t ( 
		\nabla\sigma^\delta(\nabla h^\delta(s))-\nabla\sigma^\delta(\nabla f),
		\nabla h^\delta(s)-\nabla f)_{L^2(D)^d}\,ds \\
		&\qquad {}-2\int_0^t ( 
		\nabla\sigma^\delta(\nabla f),
		\nabla h^\delta(s)-\nabla f)_{L^2(D)^d}\,ds \\
		&\le \|h_0-f\|_{H}^2-2c_{-}\int_0^t  
		\left\|\nabla h^\delta(s)-\nabla f\right\|^2_{L^2(D)^d}\,ds \\
		&\qquad {}+2\int_0^t 
		\|\nabla\sigma^\delta(\nabla f)\|_{L^2(D)^d}
		\|\nabla h^\delta(s)-\nabla f\|_{L^2(D)^d}\,ds \\
		&\le \|h_0-f\|_{H}^2-c_{-}\int_0^t  
		\left\|\nabla h^\delta(s)-\nabla f\right\|^2_{L^2(D)^d}\,ds \\
		&\qquad {}+c_{-}^{-1}\int_0^t 
		\|\nabla\sigma^\delta(\nabla f)\|^2_{L^2(D)^d}\,ds
	\end{align*}}
	from \eqref{eq2.32a} with $g=h^\delta-f$ and $\sigma$
	replaced by $\sigma^\delta$ and \eqref{convex-reg}. This shows
	\begin{equation}\label{PDE-unif-bound}
		\sup_{\delta>0}\int_0^t\|h^\delta(s)\|_{H^1(D)}^2\,ds<\infty.
	\end{equation}
	By a similar calculation, we obtain 
	\begin{align*}
		\|&h(t)-h^\delta(t)\|_{H}^2 \\
		&=-2\int_0^t ( 
		\nabla\sigma(\nabla h(s)+\nabla f)
		-\nabla\sigma(\nabla h^\delta(s)+\nabla f),
		\nabla h(s)-\nabla h^\delta(s))_{L^2(D)^d}\,ds \\
		&\quad {}-2\int_0^t ( 
		\nabla\sigma(\nabla h^\delta(s)+\nabla f)
		-\nabla\sigma^\delta(\nabla h^\delta(s)+\nabla f),
		\nabla h(s)-\nabla h^\delta(s))_{L^2(D)^d}\,ds \\
		&\le 2\int_0^t c_{+}\delta
		\int_D|\nabla h(s,\theta)-\nabla h^\delta(s,\theta)|\,d\theta\,ds.
	\end{align*}
	Note that we have used \eqref{approx-reg} in the last inequality.
	Combining the above with \eqref{PDE-unif-bound} and $h-f\in L^2([0,T],V)$,
	we get the conclusion.
\end{proof}

We can show the following proposition by a similar argument to the proof of Proposition~\ref{prop-PDE-approx} noting Lemma~\ref{lem2.13}-(1).
\begin{prop}\label{prop-PDE-approx2}
	Let $h$ and $\hat{h}$ be two solutions of \eqref{PDE1}.
	We then have 
	\[\|h(t)-\hat{h}(t)\|_{H}\le \|h(0)-\hat{h}(0)\|_{H},\quad t\ge 0.\]
\end{prop}

\subsection{The discretization for the macroscopic equation}
\label{subsec-discretized-pde}
In order to introduce the discretized equation corresponding to
the regularized macroscopic equation \eqref{PDE1-reg}, let us introduce several notations.
We define the finite difference operators by
{\allowdisplaybreaks
\begin{gather*}
	\nabla^N_i  f(\theta)=N(f(\theta+e_i /N)-f(\theta)), \\
	\nabla^{N,*}_i  f(\theta)=-N(f(\theta)-f(\theta-e_i /N)), \\
	\nabla^N f(\theta)=(\nabla^N_1 f(\theta),\dots,\nabla^N_d f(\theta)), \\
	\dive_N g(\theta)=-\sum_{i =1}^d \nabla^{N,*}_i  g_i (\theta)
\end{gather*}}%
for $f:\real^d\to\real,\,g=(g_i )_{1\le i \le d}:
\real^d\to\real^d,1\le i \le d$ and $\theta\in\real^d$,
where $e_i\in\integer^d$ is the $i$-th unit vector given by $(e_i)_j=\delta_{ij}$ for $1\le i,j\le d$.
We also define the discretized Laplacian with the Neumann boundary condition
by
\begin{align*}
	\Delta_N f(\theta)&=
	N\sum_{i =1}^d \Bigl(
	\nabla^N_i  f(\theta)1_{\theta\in \tilde{D}_N}1_{\theta+e_i/N\in \tilde{D}_N}
	\\
	&\qquad\qquad\qquad {}-\nabla^{N}_i  f(\theta-e_i/N)1_{\theta\in \tilde{D}_N}1_{\theta-e_i/N\in \tilde{D}_N}
	\Bigr)
\end{align*}
for $f:\real^d\to\real$, where the domain $\tilde{D}_N$ is defined by
\[\tilde{D}_N=\bigcup_{x\in D_N}B(x/N,1/N).\]
Note that indicator functions appearing above are corresponding to 
the range of $y$'s where the sum \eqref{discrete-Laplacian} is taken.
With these notations the discretized PDE for \eqref{PDE1-reg} reads
a system of ordinary differential equations
for $\bar{h}^{N,\delta}(t,\theta)$:
\begin{equation}\label{discretized_PDE}
\left\{\begin{split}
	\frac{\partial}{\partial t}\bar{h}^{N,\delta}(t,\theta)&=-\Delta_N k^{N,\delta} \\
	k^{N,\delta}&=A^{N,\delta}(\bar{h}^{N,\delta}(t))(\theta) \\
	&:=\dive_N\bigl\{(\nabla\sigma^\delta)(
	\nabla^N \bar{h}^{N,\delta}(t))\bigr\}(\theta),\quad
	\theta\in \tilde{D}_N, \\
	\bar{h}^{N,\delta}(t,\theta)&=0,\quad \theta\not\in\tilde{D}_N.
\end{split}\right.
\end{equation}
We recall that $\sigma^\delta\in C^\infty(\real^d)$ and it satisfies
\eqref{convex-reg}.
The equation \eqref{discretized_PDE} will be solved with the initial
data given by
\begin{equation}\label{initial_PDE}
	\bar{h}_0^{N}(\theta)=N^d\int_{B(x/N,1/N)}h_0(\theta')\,d\theta',
	\quad \theta\in B(x/N,1/N)\text{ for some }x\in\integer^d,
\end{equation}
where $h_0\in L^2(\real^d)$. Since the initial data $h_0^{N}$ is
a step function, the solution $\bar{h}^N(t,\theta)$ is also a step function,
that is,
\[
	\bar{h}^{N,\delta}(t,\theta)=\bar{h}^{N,\delta}(t,x/N),\quad \theta\in B(x/N,1/N),\,x\in\integer^d.
\]
 
\subsection{The discrete analog of the space $H$}
We shall analogously define the discrete version of
the Hilbert space $H$ introduced in Section~\ref{sec-sobolev}.
For a step function $f^N:\real^d\to\real$ with mesh size $1/N$,
that is, $f^N$ having the following representation
\begin{equation}\label{def-stepfunc}
	f^N(\theta)=\sum_{x\in\integer^d}N^{-1}\psi^N(x)1_{B(x/N,1/N)}(\theta)
\end{equation}
with $\psi^N\in\real^{\integer^d}$ such that $\psi^N(x)=0$ for $x\in\integer^d\smallsetminus D_N$, we define $\|f^N\|_{-1,N}$ by
\begin{align*}
	\|f^N\|_{-1,N}^2&=N^{-d-4}\sum_{x\in D_N}(\psi^N(x)-\langle\psi^N\rangle)
	(-\Delta_{D_N})^{-1}(\psi^N(x)-\langle\psi^N\rangle) \\
	&\qquad {}+N^{-2d-2}\langle\psi^N\rangle^2,
\end{align*}
where
\[
	\langle\psi^N\rangle=N^{-d}\sum_{x\in D_N}\psi^N(x).
\]
Note that the inverse of the Laplacian $(-\Delta_{D_N})^{-1}$ can be defined
as the linear operator from
\[
	\mathscr{A}_N=\left\{
	\phi\in\real^{D_N};
	\,\sum_{x\in D_N}\phi(x)=0
	\right\}
\]
to itself, by regarding $-\Delta_{D_N}$ as a bijection from
$\mathscr{A}_N$ to itself. It is verified by 
the summation-by-parts formula for $-\Delta_{D_N}$ as follows:
\begin{align}
	\sum_{x\in D_N}&(-\Delta_{D_N})\psi(x)\phi(x) \nonumber \\
	&=\frac{1}{2}\sum_{x\in D_N}\sum_{y\in D_N;|x-y|=1}
	\left(\psi(y)-\psi(x)\right)
	\left(\phi(y)-\phi(x)\right)  \nonumber \\
	&=\sum_{i=1}^d\sum_{x\in D_N\cap (D_N-e_i)}
	\left(\psi(x+e_i)-\psi(x)\right)
	\left(\phi(x+e_i)-\phi(x)\right) 
	\label{sum-by-parts}
\end{align}
holds for every $\phi,\psi\in\real^{D_N}$.

Under these settings, for $\psi^N\in \mathscr{A}_N$,
$\zeta^N:=(-\Delta_{D_N})^{-1}\psi^N$ is defined only on $D_N$.
Let us extend $\xi^N$ to outside of $D_N$ for convenience.
We define $\partial_i D_N$ by
\[
	\partial_i D_N=\{x\in\integer^d\smallsetminus D_N;\, \mathrm{dist}_{\integer^d}(x,D_N)=i\},\quad i\ge1,
\]
where $\mathrm{dist}_{\integer^d}$ is the graph distance on $\integer^d$.
We define the value of $\xi$ on $\partial_iD_N$ inductively, by
\[
	\zeta^N(x)=\left(\#\{y\in D_N; |x-y|=1\,\}\right)^{-1}
	\sum_{y\in D_N; |x-y|=1}\zeta^N(y)
\]
for $x\in \partial_1D_N$ and
\begin{align*}
	\zeta^N(x)&=\left(\#\{y\in \partial_{i-1}D_N; |x-y|=1\,\}\right)^{-1}
	\sum_{y\in \partial_{i-1}D_N; |x-y|=1}\zeta^N(y)
\end{align*}
for $x\in \partial_iD_N$. 
Let us take $K>0$ independent of $N$ such that
\[
	D\subset \bigcup_{x\in D_N\cup\bigcup_{i=1}^K\partial_i D_N}B(x/N,1/N).
\]
Extending $\zeta^N$ to $D_N\cup\bigcup_{i=1}^{K+1}\partial_i D_N$, we can 
introduce the macroscopic function $g^N$ on $D$ by
\begin{equation}\label{def-stepfunc2}
	G_Nf^N(\theta)=\sum_{x\in D_N\cup\bigcup_{i=1}^K\partial_i D_N}
	N^{-1}\zeta^N(x)1_{B(x/N,1/N)}(\theta),
	\quad \theta\in D.
\end{equation}
We note that $k^{N,\delta}$ in the equation \eqref{discretized_PDE}
can be extended to the function on $D$ by the same way.
We also note that we 
can choose a constant $C>0$ independent of $N$ such that
\begin{align}
	\sum_{x\in D_N\cup\bigcup_{i=1}^K\partial_i D_N}
	&\sum_{y\in\integer^d:|x-y|=1}(\zeta^N(x)-\zeta^N(y))^2 \nonumber \\
	&\le C
	\sum_{x\in D_N}
	\sum_{y\in D_N:|x-y|=1}(\zeta^N(x)-\zeta^N(y))^2, \label{eq-d3}
\end{align}
by Assumption~\ref{ass-domain}.

As a preparation for calculations, we shall show 
the following proposition, which means that $\|f^N\|_{-1,N}$ dominates
$\|f^N\|_{H^{1}(D)^*}$ if $f^N$ is a step function.
\begin{prop}\label{prop-sobolev-norms2}
		There exists a constant $C>0$ independent of $N$ such that
		\begin{equation}\label{sobolev-norms}
		\|f^N\|_{H^1(D)^*}\le C\|f^N\|_{-1,N}
		\end{equation}
		holds for every step function $f^N$ with mesh size $1/N$
		satisfying $\langle f^N,1\rangle=0$ and $f^N(x/N)=0$ for $x\in\integer^d\smallsetminus D_N$.
\end{prop}
Before starting the proof of Proposition~\ref{prop-sobolev-norms2},
we prepare a small lemma.
\begin{lemma}\label{lem-a1}
		For $\psi\in \mathscr{A}_N$ and $\phi\in\real^{D_N}$, we have
		\begin{align*}
		&\sum_{x\in D_N}\psi(x)\phi(x)\\
		&\quad =\frac{1}{2}\sum_{x\in D_N}\sum_{y\in D_N;|x-y|=1}
		\left((-\Delta_{D_N})^{-1}\psi(y)-(-\Delta_{D_N})^{-1}\psi(x)\right) 
		(\phi(y)-\phi(x)).
		\end{align*}
		Especially, we have
		\begin{align*}
		&\|f^N\|_{-1,N}^2 \\
		&\quad =\frac{1}{2}N^{-d-4}\sum_{x\in D_N}\sum_{y\in D_N;|x-y|=1}
		\left((-\Delta_{D_N})^{-1}\psi(y)
		-(-\Delta_{D_N})^{-1}\psi(x)\right)^2
		\end{align*}
		for every the step function $f^N$ represented by \eqref{def-stepfunc}
		with $\psi\in \mathscr{A}_N$.
\end{lemma}
\begin{proof}
		Take $\psi\in \mathscr{A}_N$ and $\phi\in\real^{D_N}$.
		By definition of $(-\Delta_{D_N})^{-1}$, we obtain
		\begin{align*}
			\sum_{x\in D_N}\psi(x)\phi(x)
			&=\sum_{x\in D_N}\left((-\Delta_{D_N})(-\Delta_{D_N})^{-1}\psi\right)(x)\phi(x) \\
			&=\frac{1}{2}\sum_{x\in D_N}\sum_{y\in D_N;|x-y|=1}
	\left((-\Delta_{D_N})^{-1}\psi(y)-(-\Delta_{D_N})^{-1}\psi(x)\right)\\
	&\qquad\qquad\qquad\qquad\qquad\qquad\times 
	\left(\phi(y)-\phi(x)\right) 
		\end{align*}
		which shows the conclusion.	 Here, we have used
		the summation-by-parts formula \eqref{sum-by-parts}.
\end{proof}

	\begin{proof}[Proof of Proposition~\ref{prop-sobolev-norms2}]
		We take $\psi^N$ satisfying \eqref{def-stepfunc}.
		We also take a function $J\in C^\infty(D)$ arbitrarily.
		Since $f^N(x/N)=0$ holds for $x\in\integer^d\smallsetminus D_N$,
		we have
		\begin{align*}
			\langle f^N, J\rangle
			&=\sum_{x\in D_N}\int_{B(x/N,1/N)}f^N(\theta)J(\theta)\,d\theta
		\end{align*}
		by the definition of $\langle\cdot,\cdot\rangle$. 
		Letting
		\[\xi^N(x)=N^{d+1}\int_{B(x/N,1/N)}J(\theta)\,d\theta,\quad x\in \integer^d,\]
		we then obtain
		\begin{align*}
			\langle f^N, J\rangle
			&=N^{-d-2}\sum_{x\in D_N}\sum_{y\in D_N;|x-y|=1}
		\left((-\Delta_{D_N})^{-1}\psi^N(y)-(-\Delta_{D_N})^{-1}\psi^N(x)\right) \\
		&\qquad \qquad \qquad \qquad \qquad \qquad\qquad \times
		(\xi^N(y)-\xi^N(x))
		\end{align*}
		from Lemma~\ref{lem-a1}. Using the Schwarz inequality
		and Lemma~\ref{lem-a1}, we get 
		\begin{align}
			\left|\langle f^N, J\rangle\right|
			&\le \|f^N\|_{-1,N}
			\left(N^{-d}\sum_{x\in D_N}\sum_{y\in D_N;|x-y|=1}
			\left(\xi^N(y)-\xi^N(x)\right)^2\right)^{1/2} \nonumber \\
			&=:\|f^N\|_{-1,N} I(\xi^N)^{1/2}. \label{eq-a2}
		\end{align}
		Since we have 
		\[\bigcup_{x\in D_N}\bigcup_{y\in D_N;|x-y|\le 1}B(y/N,1/N)\subset D\]
		by the definition of $D_N$ and we have
		\begin{align*}
			\left(\xi^N(x\pm e_i)-\xi^N(x)\right)^2
			& \le N^{d+1}
			\int_{B(x/N,1/N)}\int_0^{1/N}\left(\frac{\partial J}{\partial \theta_i}(\theta+\pm te_i)\right)^2dt\,d\theta
		\end{align*}
		for every $x\in D_N$ and $1\le i\le d$,
		we obtain
		\begin{align*}
			I(\xi^N)
			&\le N\sum_{x\in D_N}\sum_{i=1}^d
			\int_{B(x/N,1/N)}\int_0^{1/N}\left(\frac{\partial J}{\partial \theta_i}(\theta+te_i)\right)^2dt\,d\theta \\
			&\quad +N\sum_{x\in D_N}\sum_{i=1}^d
			\int_{B(x/N,1/N)}\int_0^{1/N}\left(\frac{\partial J}{\partial \theta_i}(\theta-te_i)\right)^2dt\,d\theta \\
			&\le 2\|\nabla J\|^2_{L^2(D)^d}.
		\end{align*}
		Combining the above with \eqref{eq-a2}, we obtain
		\[\left|\langle f^N, J\rangle\right|
		\le \sqrt{2}\|f^N\|_{-1,N}\|J\|_{H^1(D)}.\]
		The above implies the conclusion, since $C^\infty(D)$ is dense
		in $H^1(D)$.
	\end{proof}

Furthermore, the norm $\|\cdot\|_{-1,N}$ converges to $\|\cdot\|_H$
in the following sense:
\begin{prop}\label{prop-d1}
	For $N\ge1$, let $f^N$ be a step function
	represented by \eqref{def-stepfunc}. 
	If the sequence $\{f^N\}$ satisfies
	\begin{equation}\label{ass-d1-1}
		\sup_{N\ge1}\|f^N\|_{L^2(D)}<\infty
	\end{equation}
	and
	\begin{equation}\label{ass-d1-2}
		\lim_{N\to\infty}\|f^N-f\|_{H}=0
	\end{equation}
	for some $f\in H$,  
	we then have
	\[\lim_{N\to\infty}\|f^N\|_{-1,N}=\|f\|_{H}.\]
\end{prop}
\begin{proof} 
	We first obtain
	\begin{align*}
		\|f^N\|_{-1,N}^2
		&=N^{-d-4}\sum_{x\in D_N}\psi^N(x)
		(-\Delta_{D_N})^{-1}\psi^N(x)  \\
		&=(f^N,G_Nf^N)_{L^2(D)}
	\end{align*}
	by the definition of $\|\cdot\|_{-1,N}$, where
	$\psi^N$ is determined by \eqref{def-stepfunc}.
	Here, we have used $\psi^N(x)=0$ for every
	$x\in\integer^d\smallsetminus D_N$ at the last identity.
	Since we also have
	\[\|f\|^2_{H}=\langle f,Gf\rangle,\]
	we therefore obtain
	\begin{align*}
		\|f^N\|^2_{-1,N}-\|f\|^2_{H} 
		&=(f^N,G_Nf^N)_{L^2(D)}-(f^N,Gf)_{L^2(D)}\\
		&\quad  {}+\langle f^N,Gf\rangle
		-\langle f,Gf\rangle. 
	\end{align*}
	Since we have
	\[\lim_{N\to\infty}\left(\langle f^N,Gf\rangle
		-\langle f,Gf\rangle\right)=0\]
	from \eqref{ass-d1-2}, it is sufficient for our goal
	to show
	\begin{equation}\label{eq-d5}
	\lim_{N\to\infty}\left((f^N,g^N)_{L^2(D)}-(f^N,Gf)_{L^2(D)}\right)=0.
	\end{equation}
	Furthermore, once we have
	\begin{equation}\label{eq-d4}
	\lim_{N\to\infty}\|G_Nf^N-Gf\|_{L^2(D)}=0,
	\end{equation}
	we immediately obtain \eqref{eq-d5} by \eqref{ass-d1-1}.
	
	Let us show \eqref{eq-d4}.
	We take $J\in C^\infty(D)$ arbitrarily and define $\xi^N$ and $J^N$
	by
	\[\xi^N(x)=N^{d+1}\int_{B(x/N,1/N)}J(\theta')\,d\theta
	\]
	for $x\in \integer^d$ such that $B(x/N,1/N)\subset D$, and
	$J^N$ is defined by
	\[J^N(\theta)=N^{d}\int_{B(x/N,1/N)}J(\theta')\,d\theta'\]
	for $\theta\in B(x/N,1/N)$ with $x\in \integer^d$ such that $B(x/N,1/N)\subset D$, respectively.
	We note that  Lemma~\ref{lem-a1} implies
	\allowdisplaybreaks\begin{align}
	&\int_{D}f^N(\theta)J(\theta)\,d\theta \nonumber\\
	&=\sum_{i=1}^d\int_{\tilde{D}_N}\nabla^{N}_i G_Nf^N(\theta)\nabla^{N}_iJ^N(\theta)d\theta \nonumber\\
	&\quad {}-N^{-d-4}\sum_{i=1}^d\sum_{x\in D_N\cap (D_N-e_i)^\complement}
	(\zeta^N(x+e_i)-\zeta^N(x))(\xi^N(x+e_i)-\xi^N(x)),
	\label{eq-d1}
	\end{align}
	where $\zeta^N$ is the extension
	of $(-\Delta_{D_N})^{-1}\psi^N$, see the beginning of this subsection.
		
	Using \eqref{eq-d3} and Lemma~\ref{lem-a1}, 
	we can choose a subsequence $\{N'\}$ such that
	\begin{align*}
		G_Nf^N&\to\bar{g}\quad \text{strongly in $L^2(D)$,} \\
		\nabla^{N'}_i G_Nf^N&\to h_i\quad \text{weakly in $L^2(D)$} 
	\end{align*}
	as $N'\to\infty$ for some $\bar{g},h_i\in L^2(D)$.
	Replacing $N$ by $N'$ at \eqref{eq-d1} and 
	taking the limit $N'\to\infty$ at the both sides of \eqref{eq-d1},
	we obtain
	{\allowdisplaybreaks\begin{align*}
	\int_{D}f(\theta)J(\theta)\,d\theta&=
	\sum_{i=1}^d\int_{D}h_i(\theta)
	\nabla_iJ(\theta)d\theta.
	\end{align*}}
	Since the limit point $\bar{g}$ satisfies
	\[\int_{D}\bar{g}(\theta)\,d\theta=0\]
	and $\nabla_i\bar{g}=h_i$ for every $1\le i\le d$, we conclude that
	every limit point $\bar{g}$ satisfies \eqref{def-green1} and
	\eqref{def-green2}.
	Because the solution of the elliptic equation 
	\eqref{def-green1} and \eqref{def-green2} exists uniquely,
	the sequence $\{G_Nf^N\}$ itself need to converge to $\bar{g}=Gf$ in $L^2(D)$.
\end{proof}

\subsection{A priori bound for the discretized equation}
\label{subsec-apriori-pde}
Let us establish a priori bound of the solution $\bar{h}^{N,\delta}$
of \eqref{discretized_PDE}. 
To do so, we introduce an auxiliary function
similarly to Definition~\ref{def-PDE}.
Let us take the function $g\in C_0^\infty(D)$ which satisfies the following:
\begin{enumerate}
	\item $g(x)\ge0$ holds for every $x\in D$.
	\item $\int_D g(\theta)\,d\theta=1$.
\end{enumerate} 
We can then take $N_0\ge 1$ large enough such that
$\supp g\subset \tilde{D}_{N}$ for every $N\ge N_0$.
Let us establish a priori bound for $\bar{h}^N$ with $N\ge N_0$.

Using $g$ introduced above, we define $\zeta^N$ by
\[
\zeta^N(x)=N^{d+1}\int_{B(x/N,1/N)}g(\theta)d\theta,\quad x\in\integer^d
\]
and $g^N$ by
\[g^N(\theta)=\sum_{x\in\integer^d}N^{-1}\zeta^N(x)1_{B(x/N,1/N)}(\theta),
\quad\theta\in\real^d.\]
We then have the following bound:
\begin{equation}\label{bound_aux}
	\sup_{N\ge 1}\sup_{1\le i,j\le d}
	\left\{
	\left\|g^N\right\|_{\infty}+
	\left\|\nabla^N_i g^N\right\|_{\infty}+
	\left\|\nabla^N_i\nabla^N_j g^N\right\|_{\infty}
	\right\}\le c_g,
\end{equation}
where $c_g$ is the constant defined by
\[
	c_g=\sup_{1\le i,j\le d}
	\left\{
		\|g\|_{\infty}+
		\left\|\frac{\partial g}{\partial \theta_i}\right\|_{\infty}+
		\left\|\frac{\partial^2 g}{\partial \theta_i\theta_j}\right\|_{\infty}
	+1\right\}.
\]
Using $\zeta^N$ and $g^N$ introduced above, we define $\psi^N$ and $f^N$ by
\begin{equation}\label{eq2.3.1}
\psi^N(x)=v\zeta^N(x),\quad x\in\integer^d
\end{equation}
and
\begin{equation}\label{eq2.3.1b}
f^N(x)=vg^N(\theta),\quad \theta\in\real^d
\end{equation}
respectively, where $v$ is the ``volume'' of $h_0$, that is,
\[
v=\int_D \bar{h}^{N}_0(\theta)\,d\theta=\langle h_0,1\rangle
\]
for the initial datum $\bar{h}_0^N$ for \eqref{discretized_PDE}.
Note that the right hand side of the above does not depend on the
choice of $\delta$.
We then have that sequences $\{\psi^N\}$ and $\{f^N\}$ satisfy
following properties:
\begin{enumerate}
	\item For every $N$ and $x\in \integer^d\smallsetminus D_N$,
	$\psi^N(x)=f^N(x/N)=0$ holds.
	\item For every $N$, the relationship
	\[
	\sum_{x\in D_N}\psi^N(x)
	=N\int_D f^N(\theta)\,d\theta
	\]
	holds.
	\item The following bounds hold:
	\begin{equation}
	\sup_{N\ge 1}\sup_{1\le i,j\le d}
	\left\{
	\left\|f^N\right\|_\infty+
	\left\|\nabla^N_i f^N\right\|_\infty+
	\left\|\nabla^N_i\nabla^N_j f^N\right\|_\infty
	\right\}\le c_g|v|.
\end{equation}
\end{enumerate}

\begin{prop}
	There exist constants $C_1,C_2>0$ independent of $N,\delta$
	such that
	\[\|\bar{h}^{N,\delta}(t)\|^2_{-1,N}+N^{-d}\int_0^t
	\|\nabla^N \bar{h}^{N,\delta}(s)\|_{L^2(D)^d}^2\,ds
	\le C_1\|h^{N,\delta}(0)\|^2_{-1,N}+C_2(1+t)\]
	holds for every $t\ge0$.
\end{prop}
\begin{proof}
	Differentiating $\|\bar{h}^{N,\delta}(t)-f^N\|_{-1,N}^2$ in $t$, we obtain
\begin{align}
	\frac{d}{dt}
	\|\bar{h}^{N,\delta}(t)-f^N\|^2_{-1,N} 
	&=-2N^{-d}\sum_{x\in \overline{D_N}}
	\nabla^N \bar{h}^{N,\delta}(t,x/N)
	\cdot\nabla\sigma^{\delta}(\nabla^N \bar{h}^{N,\delta}(t,x/N)) \nonumber \\
	&\qquad {}+2N^{-d}\sum_{x\in \overline{D_N}}
	\nabla^N f^N(x/N)\cdot\nabla\sigma^\delta(\nabla^N \bar{h}^{N,\delta}(t,x/N)) 
	\nonumber \\
	&\le -2c_{-}N^{-d}\sum_{x\in \overline{D_N}}
	\left|\nabla^N \bar{h}^{N,\delta}(t,x/N)\right|^2 \nonumber \\
	&\qquad {}+2N^{-d}\sum_{x\in \overline{D_N}}
	\nabla^N f^N(x/N)\cdot\nabla\sigma^\delta(\nabla^N \bar{h}^{N,\delta}(t,x/N)),
	\label{eq2.8}
	\end{align}
	where $\overline{D_N}$ is defined by
	\[\overline{D_N}=\{x\in\integer^d;\, \text{there exists $y\in D_N$
	such that $|x-y|\le 1$}\}.\]
	Here, we have used the summation-by-parts formula
	\begin{equation}\label{eq2.10b}
	\sum_{x\in D_N}\alpha(x/N)\dive_N \beta(x/N)
	=-\sum_{x\in\overline{D_N}}\nabla^N\alpha(x/N)\cdot\beta(x/N),
	\end{equation}
	where $\alpha:\real^d\to\real$ 
	and $\beta=(\beta_i)_{i=1}^d:\real^d\to\real^d$
	are arbitrary functions such that
	$\alpha$ and $\beta_i\,(1\le i\le d)$ are step functions
	with mesh size $1/N$ and $\alpha(x/N)=0$
	for every $x\in \integer^d\smallsetminus D_N$.

	The second term in the right hand side can be estimated in the following
	way:
	\begin{align*}
		\Biggl|N^{-d}\sum_{x\in \overline{D_N}}&
	\nabla^N f^N(x/N)\cdot\nabla\sigma^\delta(\nabla^N \bar{h}^{N,\delta}(t,x/N)) \Biggr| \\
	&\le \frac{1}{2}c_{-}N^d\sum_{x\in \overline{D_N}}\left|\nabla^N\bar{h}^{N,\delta}(t,x/N)\right|^2
	+CN^{-d}|v|\Bigl|\overline{D_N}\Bigr|,
	\end{align*}
	with a constant $C>0$ independent of $N$.
	We have used the properties of $f^N$
	stated at the beginning of this subsection.
	Plugging the above into \eqref{eq2.8} and integrating in $t$, we obtain
	\begin{align*}
		\|\bar{h}^{N,\delta}(T)-f^N\|^2_{-1,N}
		&\le \|\bar{h}^{N,\delta}(0)-f^N\|^2_{-1,N} \\
		&\qquad {}-\frac{1}{2}c_{-}\int_0^T\|\nabla^N\bar{h}^{N,\delta}(t)\|_{L^2(D)^d}^2\,dt
		+CN^{-d}|v|\Bigl|\overline{D_N}\Bigr|T
	\end{align*}
	for every $T>0$,
	which implies the desired estimate, since $\|f^N\|_{-1,N}$ is
	bounded uniformly in $N$.
\end{proof}

We can improve the bound for $\nabla^N\bar{h}^{N,\delta}$ if the initial datum
is smooth enough. 
\begin{prop}\label{prop2.2}
	We assume that
	\begin{equation}\label{smoothness-initial}
		h_0\in C_0^\infty(D).
	\end{equation}
	We then have the following uniform bound:
	\[\sup_{N}\sup_{0\le t\le T}\|\nabla^N\bar{h}^{N,\delta}(t)\|^2_{L^2(D)^d}<\infty\]
	for every $T>0$.
\end{prop}
\begin{proof}
	Differentiating 
	$\sum_{x\in\overline{D_N}}\sigma^\delta(\nabla^N \bar{h}^{N,\delta}(t,x/N))$ in $t$,
	we have
	\begin{align}
		\frac{\partial}{\partial t}
		&\sum_{x\in\overline{D_N}}\sigma^\delta(\nabla^N \bar{h}^{N,\delta}(t,x/N))	\nonumber \\
		&=\sum_{x\in \overline{D_N}}
		\nabla\sigma^\delta(\nabla^N \bar{h}^{N,\delta}(t,x/N))
		\nabla^N\frac{\partial}{\partial t}\bar{h}^{N,\delta}(t,x/N) \nonumber \\
		&=\sum_{x\in D_N}
		\dive_N\nabla\sigma^\delta(\nabla^N \bar{h}^{N,\delta}(t,x/N))
		\frac{\partial}{\partial t}\bar{h}^{N,\delta}(t,x/N) \nonumber \\
		&=\sum_{x\in D_N}
		k^{N,\delta}(t,x/N)\Delta_N k^{N,\delta}(t,x/N). \label{eq2.9}
	\end{align}
	Here, we have used (3.8) in \cite{N03}, since
	\[\frac{\partial \bar{h}^N}{\partial t}(t,x/N)=0,\quad x\in D_N^\complement\]
	holds.
	Since $-\Delta_N$ is non-negative definite, we obtain
	that the right hand side is non-positive.
	Dropping the right hand side and integrating in $t$, we have
	\[
		\sum_{x\in\overline{D_N}}\sigma^\delta(\nabla^N \bar{h}^{N,\delta}(t,x/N))
		\le 
		\sum_{x\in\overline{D_N}}\sigma^\delta(\nabla^N \bar{h}^{N,\delta}(0,x/N)),	\]
	which indicate the conclusion, since the function $\sigma$
	satisfies \eqref{eq3.1a}.
\end{proof}

Let us establish the bound for $k^N$ in \eqref{discretized_PDE},
in order to apply the argument in Section~3.3 of \cite{N03}.
\begin{prop}\label{prop2.3}
	We assume \eqref{smoothness-initial}. We then have the following bound:
	\[
	\sup_N\left\{
	\sup_{0\le t\le T}\left\|\frac{\partial \bar{h}^{N,\delta}}{\partial t}
	\right\|_{-1,N}^2
	+\int_0^T
	\left\|\nabla^N\frac{\partial \bar{h}^{N,\delta}}{\partial t}\right\|^2_{L^2(D)^d}\,dt
	\right\}<\infty.
	\]
\end{prop}
\begin{proof}
	Noting 
	\[
	\frac{\partial^2 \bar{h}^{N,\delta}}{\partial t^2}(t,x/N)
	=\sum_{i,j=1}^d\Delta_N \nabla_i^{N,*}
	\left\{
	\frac{\partial^2\sigma^{\delta}}{\partial u_i\partial u_j}
	(\nabla^N \bar{h}^{N,\delta})
	\nabla^N_j\frac{\partial \bar{h}^{N,\delta}(t)}{\partial t}
	\right\}
	\]
	and 
	\[\sum_{x\in D_N}\frac{\partial \bar{h}^{N,\delta}}{\partial t}(t,x/N)=0,\]
	we have
\begin{align}
	\frac{d}{dt}&
	\left\|\frac{\partial \bar{h}^{N,\delta}}{\partial t}\right\|^2_{-1,N} \nonumber\\
	&=-2N^{-d}
	\sum_{i,j=1}^d\sum_{x\in \overline{D_N}}
	\nabla^N_i\frac{\partial \bar{h}^{N,\delta}}{\partial t}(t,x/N)
	\frac{\partial^2\sigma}{\partial u_i\partial u_j}
	(\nabla^N \bar{h}^{N,\delta})\nabla^N_j\frac{\partial \bar{h}^{N,\delta}(t)}{\partial t}
	\nonumber \\
	&\le -2c_{-}N^{-d}
	\sum_{i=1}^d\sum_{x\in \overline{D_N}}
	\left(\nabla^N_i\frac{\partial \bar{h}^{N,\delta}}{\partial t}(t,x/N)\right)^2
	\nonumber
	\end{align}
	by performing the summation-by-parts several times.
	Integrating the both sides in $t$, we obtain the conclusion.
\end{proof}
\begin{remark}
	By the definition of $k^{N,\delta}$, we have
	\[\left\|\frac{\partial \bar{h}^{N,\delta}}{\partial t}
	\right\|_{-1,N}^2=N^{-d}\sum_{x\in D_N}
		k^{N,\delta}(t,x/N)(-\Delta_Nk^{N,\delta})(t,x/N).\]
	We therefore obtain
	\begin{equation}\label{eq2.12dd}
		\sup_{N}\sup_{0\le t\le T}N^{-d}\sum_{x\in D_N}
		k^{N,\delta}(t,x/N)(-\Delta_Nk^{N,\delta})(t,x/N)<\infty.
	\end{equation}
	by Proposition~\ref{prop2.3}.
\end{remark}
\begin{remark}
	We need the smoothness of $\sigma^\delta$
	in order to obtain the uniform bound in $N$
	for
	\[
	N^{-d}\sum_{x\in D_N}
	k^{N,\delta}(0,x/N)(-\Delta_Nk^{N,\delta})(0,x/N),\]
	even if $h_0$ is smooth, for example, $h_0\in C_0^\infty(D)$.
	We have $C^1$-regularity of $\sigma$
	and the Lipschitz continuity of $\nabla\sigma$,
	but such regularity is less than that we need.
	It is the reason why we consider
	the equation \eqref{discretized_PDE} with smooth $\sigma^\delta$
	instead of the original surface tension $\sigma$.
\end{remark}
As a direct consequence of Proposition~\ref{prop2.3},
we have the following result.
\begin{cor}\label{cor2.4}
	We assume \eqref{smoothness-initial}.
	There exists a constant $C>0$ independent of $N$ such that
	\[\sup_{N}\left\|\bar{h}^{N,\delta}(t_1)-\bar{h}^{N,\delta}(t_2)\right\|^2_{-1,N}\le C|t_1-t_2|\]
	holds for every $0\le t_1,t_2\le T$.
\end{cor}

\subsection{Uniform $L^2$-bound for $k^{N,\delta}$}
In this subsection, let us establish the uniform 
$L^2$-bound for $k^{N,\delta}$. 
Noting \eqref{eq2.12dd}, we can obtain the desired bound once we show
the following:
\begin{prop}\label{prop2.7}
	Under the assumption \eqref{smoothness-initial}, we obtain
	\begin{equation}
		\sup_N \sup_{0\le t\le T}\left|\langle k^{N,\delta}(t)\rangle\right|<\infty \label{eq2.12d}
	\end{equation}
	for every $T>0$, where
	\[\langle k^{N,\delta}(t)\rangle=N^{-d}\sum_{x\in D_N}k^{N,\delta}(t,x/N).\]
\end{prop}
In order to prove the above, we shall use the same argument
as in Section~3.3 of \cite{N03} and reduce our problem
to that for solutions of elliptic equations whose main term is linear.
Repeating the argument in in Section~3.3 of \cite{N03}, we have
the following decomposition for $\nabla\sigma$:
\begin{equation}\label{decomp-surface}
	\nabla\sigma(u)=A(u)u+a(u),
\end{equation}
where $A(u)=(A_{ij}(u))$ is the $d\times d$ diagonal matrix for every
$u\in\real^d$ and $a(u)=(a_i(u)):\real^d\to\real^d$.
We remark that the matrix $A(u)$ satisfies
\begin{equation}\label{ellipticity}
c_{-}\mathbb{I}\le A(u)\le c_{+}\mathbb{I}
\end{equation}
uniformly in $u$, where $\mathbb{I}$ is the $d\times d$ identity matrix
and $c_\pm$ is same as in Assumption~\ref{ass-potential}.
Furthermore, we also remark that the vector $a$ satisfies
\[
	\sup_{1\le i\le d}\sup_{u\in\real^d}|a_i(u)|\le C_a
\]
with some constant $C_a>0$.
By \eqref{decomp-surface} and the definition of $\sigma^\delta$,
we also obtain the decomposition for $\nabla\sigma^\delta$ as follows:
\[
	\nabla\sigma^\delta(u)=A^\delta(u)u+a^\delta(u),
\]
where
\[
	A^\delta_{ij}(u)=A_{ij}*\rho_\delta(u),\quad 1\le i,j\le d
\]
and
\[
	a_i^\delta(u)=a_{ij}*\rho_\delta(u)+
	\int_{\real^d}A_{ii}(u-v)v_i\rho_\delta(v)\,dv
	,\quad 1\le i\le d.
\]
We note that $A^\delta$ is diagonal and satisfies \eqref{ellipticity} again,
and that $a^\delta$ satisfies
\[
	\sup_{\delta>0}\sup_{1\le i\le d}\sup_{u\in\real^d}|a_i^\delta(u)|<C'_a
\]
with some constant $C'_a>0$. Putting
\begin{align*}
	A^{N,\delta}(t,\theta)&=A^\delta(\nabla^N \bar{h}^{N,\delta}(t,\theta)), \\
	a^{N,\delta}(t,\theta)&=a^\delta(\nabla^N \bar{h}^{N,\delta}(t,\theta)),
\end{align*}
we see that $\bar{h}^{N,\delta}(t)$ satisfies
\begin{equation}\label{PDE_elliptic}
	k^{N,\delta}(t)=\dive_N\bigl(
	A^{N,\delta}(t)\nabla^N \bar{h}^{N,\delta}(t))\bigr)
	+\dive_Na^{N,\delta}(t),\quad t\ge0.
\end{equation}
Let us regard \eqref{PDE_elliptic} as the elliptic equation for given $A^{N,\delta}(t)
,a^{N,\delta}(t)$ and $k^{N,\delta}$, which has a unique solution.
Since the main term of \eqref{PDE_elliptic} is linear,
we can use the principle of superposition.
For $t>0$, let $\bar{h}^{N,\delta}_1(t)$ be the unique solution of 
\begin{equation}\label{PDE_elliptic1}
	\tilde{k}^{N,\delta}(t)=\dive_N\bigl(
	A^{N,\delta}(t)\nabla^N \bar{h}_1^{N,\delta}(t))\bigr)
	+\dive_Na^{N,\delta}(t)
\end{equation}
with the Dirichlet boundary condition  
\[
		\bar{h}_1^{N,\delta}(t,x/N)=0,\quad x\in \integer^d\smallsetminus D_N,
\]
where $\tilde{u}^{N,\delta}(t)$ is defined by
\[
\tilde{k}^{N,\delta}(t):=k^{N,\delta}(t)-\langle k^{N,\delta}(t)\rangle.
\]
Furthermore, we let $\bar{h}^{N,\delta}_2(t)$ be the unique solution of
\begin{equation}\label{PDE_elliptic2}
	\langle k^{N,\delta}(t)\rangle=\dive_N\bigl(
	A^{N,\delta}(t)\nabla^N \bar{h}_2^{N,\delta}(t))\bigr)
\end{equation}
with the Dirichlet boundary condition similarly to $\bar{h}_1^{N,\delta}$.
Note that the original $\bar{h}^{N,\delta}(t)$ can be expressed as
\begin{equation}\label{eq2.16c}
\bar{h}^{N,\delta}(t)=\bar{h}^{N,\delta}_1(t)+\bar{h}^{N,\delta}_2(t).
\end{equation}
We will establish \eqref{eq2.12d} via the bounds for
$\bar{h}^{N,\delta}(t),\bar{h}^{N,\delta}_1(t)$ and
$\bar{h}^{N,\delta}_2(t)$.
Let us take $\delta>0$ arbitrarily
and we sometimes omit the parameter $\delta$ throughout this subsection
in order to keep notations simple.
%
%

We shall at first show the bound for $\bar{h}_1^{N}$.
\begin{prop}\label{prop2.4}
	There exist constants $C_1,C_2>0$ independent of $N$ and $t$
	such that
	\[\|\nabla^N \bar{h}_1^N(t)\|^2_{L^2(D)^d}+
	\|\bar{h}_1^N(t)\|^2_{L^2(D)}\le C_1\|\tilde{k}^{N,\delta}(t)\|^2_{L^2(D)}+C_2\]
	holds for every $t\ge 0$.	
\end{prop}
\begin{proof}
	Multiplying the both side of \eqref{PDE_elliptic1} by
	$\bar{h}_1^N(t)$ and taking the sum over $D_N$, we have
	\begin{align*}
	\sum_{x\in D_N}\tilde{k}^{N,\delta}(t,x/N)
	\bar{h}_1^{N}(t,x/N)
	&=\sum_{x\in D_N}\bar{h}_1^{N}(t,x/N)\dive_N\bigl(
	A^N(t)\nabla^N \bar{h}_1^{N}(t))\bigr)(x/N) \\
	&\qquad +\sum_{x\in D_N}\bar{h}_1^{N}(t,x/N)\dive_Na^N(t)(x/N).
	\end{align*}	
	Dividing the both side by $N^d$ and performing the summation-by-parts,
	we obtain
	\begin{align*}
	N^{-d}&\sum_{x\in D_N}\tilde{k}^{N,\delta}(t,x/N)
	\bar{h}_1^{N}(t,x/N) \\
	&=N^{-d}\sum_{x\in \overline{D_N}}\nabla^N\bar{h}_1^{N}(t,x/N)\cdot
	A^N(t)\nabla^N \bar{h}_1^{N}(t,x/N) \\
	&\qquad +N^{-d}\sum_{x\in \overline{D_N}}\nabla\bar{h}_1^{N}(t,x/N)a^N(t,x/N)  \\
	&\le -\frac{1}{2}c_{-}\sum_{x\in \overline{D_N}}\left|\nabla^N\bar{h}_1^{N}(t,x/N)\right|^2+\frac{8C_a}{c_{-}}N^d|D_N|.
	\end{align*}
	We have used \eqref{ellipticity} and 
	$\bar{h}_1^{N}(t,x/N)=0$ for $x\in \integer^d\smallsetminus D_N$.
	Using the Poincar\'e inequality
	\begin{equation}\label{poincare-discrete}
		\|\bar{h}_1^N(t)\|^2_{L^2(D)}\le C\|\nabla^N \bar{h}_1^N(t)\|^2_{L^2(D)^d}
	\end{equation}
	with a constant $C>0$ independent in $N$,
	we have
	\begin{align*}
	\frac{1}{2}c_{-}&\|\nabla^N\bar{h}_1^{N}(t)\|_{L^2(D)^d}^2\\
	&\le 
	2\gamma \|\tilde{k}^{N,\delta}(t)\|^2_{L^2(D)}
	+2\gamma^{-1}\|\bar{h}_1^{N}(t)\|^2_{L^2(D)}
	+\frac{8C_a}{c_{-}}N^d|D_N| \\
	&\le 
	2\gamma \|\tilde{k}^{N,\delta}(t)\|^2_{L^2(D)}
	+2\gamma^{-1}C\|\nabla\bar{h}_1^{N}(t)\|^2_{L^2(D)^d}
	+\frac{8C_a}{c_{-}}N^d|D_N|
	\end{align*}
	for every $\gamma>0$. Choosing $\gamma=4C/{c_{-}}$, we conclude
	\begin{align*}
	\frac{1}{4}c_{-}\|\nabla^N\bar{h}_1^{N}(t)\|_{L^2(D)^d}^2
	&\le \frac{8C}{c_{-}} \|\tilde{k}^{N,\delta}(t)\|^2_{L^2(D)}
	+\frac{8C_a}{c_{-}}N^d|D_N|.
	\end{align*}
	Applying the Poincar\'e inequality \eqref{poincare-discrete}
	to the above, we also 
	obtain the bound for $\|\bar{h}_1^{N}(t)\|_{L^2}^2$.
\end{proof}

Since we now have the nice bounds for $\bar{h}_1^N(t)$ and $\bar{h}^N(t)$,
we have one for $\bar{h}_2^N(t)$ also. We shall show the following proposition
which says that the bound for $\bar{h}_2^N(t)$ implies the bound for $\langle k^{N,\delta}(t)\rangle$.
\begin{prop}\label{prop2.5}
	For $\alpha\in\real$, let
	$\bar{h}^N_{2,\alpha}(t)$ be the solution of
		\begin{equation}\label{PDE_elliptic2-1}
		 \alpha=\dive_N\bigl(
		A^{N}(t)\nabla^N \bar{h}_{2,\alpha}^{N}(t))\bigr)
		\end{equation}
		with the Dirichlet boundary condition. We then have
		\begin{equation}\label{eq2.17}
			\alpha^2 
			\le C\|\nabla^N\bar{h}_{2,\alpha}^N\|_{L^2(D)^d}^2.
		\end{equation}
		with a constant $C>0$ independent of $N,t$ and $\alpha$.
\end{prop}
\begin{proof}
	Multiplying \eqref{PDE_elliptic2-1} by $h_{2,\alpha}^N$
	and taking sum over $D_N$, we have
	\begin{align*}
	\alpha N^{-d}&\sum_{x\in D_N}\bar{h}_{2,\alpha}^N(t,x/N) \\
	&=N^{-d}\sum_{x\in D_N}\bar{h}_{2,\alpha}^N(t,x/N)
	\dive_N\left(A^{N}(t)\nabla^N \bar{h}_{2,\alpha}^N(t,x/N)\right).
	\end{align*}
	Performing the summation-by-parts at the right hand side,
	we obtain
	\begin{align}
	\alpha N^{-d}&\sum_{x\in D_N}\bar{h}_{2,\alpha}^N(t,x/N) \nonumber \\
	&=-N^{-d}\sum_{x\in \overline{D_N}}\nabla^N\bar{h}_{2,\alpha}^N(t,x/N)\cdot
	A^{N}(t)\nabla^N \bar{h}_{2,\alpha}^N(t,x/N). \label{eq2.16}
	\end{align}
	We have used $\bar{h}_{2,\alpha}^N(t)$ satisfies the Dirichlet
	boundary condition.
	Noting that $h_{2,\alpha}^N$ can be expressed by
	\[h_{2,\alpha}^N=\alpha h_{2,1}^N\]
	because the right hand side of \eqref{PDE_elliptic2-1} is linear,
	we get
	\begin{align*}
	\Biggl|\alpha^2 N^{-d}&\sum_{x\in D_N}\bar{h}_{2,1}^N(t,x/N)\Biggr|\\
	&=N^{-d}\sum_{x\in \overline{D_N}}\nabla^N\bar{h}_{2,\alpha}^N(t,x/N)
	\cdot A^{N}(t)\nabla^N \bar{h}_{2,\alpha}^N(t,x/N) \\
	&\le c_{+}N^{-d}\sum_{x\in \overline{D_N}}\left|\nabla^N\bar{h}_{2,\alpha}^N(t,x/N)\right|^2
	\end{align*}
	from \eqref{ellipticity}.
	Here, once we have 
	\begin{equation}\label{eq2.21}
	\inf_{N}\left|N^{-d}\sum_{x\in D_N}\bar{h}_{2,1}^N(t,x/N)\right|>c
	\end{equation}
	with a constant $c>0$ independent of $N$ and $t$,
	we immediately obtain the conclusion. We shall therefore show \eqref{eq2.21}.
	Using \eqref{eq2.16} and \eqref{ellipticity}, we get
	\begin{align}
	\left|N^{-d}\sum_{x\in D_N}\bar{h}_{2,1}^N(t,x/N)\right|
	&\ge c_{-}
	\left\|\nabla^N\bar{h}_{2,1}^N\right\|^2_{L^2(D)^d}. \label{eq2.18b}
	\end{align}
	For the function $g^N$ introduced at the beginning of Section~\ref{subsec-apriori-pde},
	we obtain
	\begin{align*}
	1&=N^{-d}\sum_{x\in D_N}g^N(x/N) \\
	&=N^{-d}\sum_{x\in \overline{D_N}}\nabla^Ng^N(x/N)\cdot
	A^{N}(t)\nabla^N \bar{h}_{2,1}^N(t,x/N) \\
	&\le C'\|\nabla^N g^N\|_{L^2(D)^d}\|\nabla^N \bar{h}_{2,1}^N(t)\|_{L^2(D)^d}.
	\end{align*}
	with a constant $C'>0$ by using \eqref{ellipticity}.
	Combining the above with \eqref{bound_aux} and \eqref{eq2.18b}, we get
	\begin{align*}
	\left|N^{-d}\sum_{x\in D_N}\bar{h}_{2,1}^N(t,x/N)\right|
	\ge \frac{c_{-}}{C'^{2}c_{g}^2},
	\end{align*}
	which shows \eqref{eq2.21}.
\end{proof}

Once we obtain Proposition~\ref{prop2.4} and Proposition~\ref{prop2.5},
we can easily show Proposition~\ref{prop2.7}.
Noting
\[
\|\nabla^N h_2^N(t)\|^2_{L^2(D)^d}\le 
2\|\nabla^N h^N(t)\|^2_{L^2(D)^d}+ 
2\|\nabla^N h_1^N(t)\|^2_{L^2(D)^d},\]
we have 
\[
		\sup_N \sup_{0\le t\le T}\|\nabla^N h_2^N(t)\|_{L^2(D)^d}<\infty
\]
by using Proposition~\ref{prop2.2} and Proposition~\ref{prop2.4}.
On the other hand, since we have
\[|\langle k^{N,\delta}(t)\rangle|^2\le C\|\nabla^N h_2^N(t)\|_{L^2(D)^d}^2\]
from Proposition~\ref{prop2.5}, we obtain \eqref{eq2.12d}
and therefore Proposition~\ref{prop2.7}.

\begin{cor}\label{cor2.7}
	Under the assumption \eqref{smoothness-initial},
	we have the following bound:
	\begin{equation}\label{eq2.23c}
		\sup_{N}\sup_{0\le t\le T}\left(\|k^{N,\delta}(t)\|_{L^2(D)}+\|\nabla^N k^{N,\delta}(t)\|_{L^2(D)^d}\right)<\infty.
	\end{equation}
\end{cor}
\begin{proof}
	From Assumption~\ref{ass-domain} and the definition of $k^{N,\delta}$
	as in Section~\ref{subsec-discretized-pde}, we can easily see
	\[\|k^{N,\delta}(t)\|_{L^2(D)}\le C_1
	\sum_{x\in D_N}|k^{N,\delta}(t,x/N)|^2\]
	and 
	\[\|\nabla^N k^{N,\delta}(t)\|^2_{L^2(D)}\le 
	C_2N^{-d}\sum_{x\in D_N}
		k^{N,\delta}(t,x/N)(-\Delta_N k^{N,\delta})(t,x/N)\]
	with constants $C_1,C_2>0$ independent of $N$. 
	These inequalities and Proposition~\ref{prop2.7} imply \eqref{eq2.23c}.
\end{proof}

Once we have Corollary~\ref{cor2.7}, we can also obtain
uniform $L^p$-bound for $\nabla^N\bar{h}^N$. This uniform bound plays the key role in the derivation of PDE \eqref{PDE1}
from the height variable $h^N$.
\begin{prop}\label{prop2.8}
	We assume \eqref{smoothness-initial}.
	We then have the following bounds:
	\begin{equation*}
	\sup_{N}\sup_{0\le t\le T}
	\|k^{N,\delta}(t)\|_{L^p(D)}^p<\infty
	\end{equation*}
	and
	\[
	\sup_N\sup_{0\le t\le T}
	\|\nabla^N \bar{h}^{N,\delta}(t)\|_{L^p(D)^d}^p<\infty
	\]
	for some $p>2$.
\end{prop}
\begin{proof}
	Combining Corollary~\ref{cor2.7} with the similar argument to the proof of
	Proposition~I.4 in \cite{FS97}, we obtain the first assertion.
	Applying the argument in Section~3.3 of \cite{N03} to \eqref{PDE_elliptic1}, we also obtain the second assertion.
\end{proof}

As an application of Proposition~\ref{prop2.8}, we can obtain the following
identity corresponding to the oscillation inequality in \cite{DGI00}.
This also plays the key role in the derivation of PDE \eqref{PDE1}
from the height variable $h^N$. 
\begin{prop}\label{oscillation}
	We assume \eqref{smoothness-initial}.
	For the solution $\bar{h}^{N,\delta}$ of \eqref{discretized_PDE} 
	and $e\in\integer^d$ such that $|e|=1$, we have
	\begin{equation}\label{eq2.23b}
		\lim_{N\to\infty}N^{-d}\int_0^T
		\sum_{x\in\overline{D_N}}\left|\nabla^N\bar{h}^{N,\delta}(t,x/N+e/N)
		-\nabla^N\bar{h}^{N,\delta}(t,x/N)\right|^2\,dt=0.
	\end{equation}
\end{prop}
\begin{proof}
	Take $1\le i\le d$ arbitrary.
	From \eqref{discretized_PDE}, we can split
	\[		F^N_0(t):=N^{-d-2}\sum_{x\in \overline{D_N}}\nabla_i^N k^{N,\delta}(t,x/N) 
		\nabla_i^N \bar{h}^{N,\delta}(t,x/N) \]
	into three terms as follows:
	\begin{align*}
		F^N_0(t) &=N^{-d-2}\sum_{x\in \overline{\overline{D_N}}}\sum_{j=1}^d
		\nabla_i^N\nabla\sigma(\nabla^N\bar{h}^{N,\delta}(t,x/N))
		\nabla_i^N\nabla_j^N\bar{h}^{N,\delta}(t,x/N) \\
		&\qquad {}-N^{-d-2}\sum_{x\in \overline{D_N}\cap (D_N-e_i)^\complement}
		\nabla_i^N\dive_N\nabla\sigma(\nabla^N\bar{h}^{N,\delta}(t,x/N))\\
		&\qquad\qquad\qquad\qquad\qquad\qquad\qquad\times
		\nabla_i^N \bar{h}^{N,\delta}(t,x/N) \\
		&\qquad {}+N^{-d-2}\sum_{x\in \overline{D_N}\cap (D_N-e_i)^\complement}
		\nabla_i^N k^{N,\delta}(t,x/N)
		\nabla_i^N \bar{h}^{N,\delta}(t,x/N) \\
		&=:F_1^N(t)+F_2^N(t)+F_3^N(t),
	\end{align*}
	where the set $\overline{\overline{D_N}}$ is defined by $\overline{\overline{D_N}}=\overline{\Bigl(\overline{D_N}\Bigr)}$.	
	Here, we have used 
	\[\nabla^N\bar{h}^{N,\delta}(t,x/N)=0,\quad x\in \integer^d\smallsetminus \overline{D_N}.\]
	Using \eqref{eq3.1a}, we have for $F_1^N(t)$
	\begin{align*}
		F_1^N(t)&\le -c_{-}N^{-d}\sum_{x\in \overline{\overline{D_N}}}
		\left|\nabla^N\bar{h}^{N,\delta}(t,x/N+e_i/N)-
		\nabla^N\bar{h}^{N,\delta}(t,x/N)\right|^2,
	\end{align*}
	which is nothing but our target. 
	From now on, we shall show the remaining terms
	vanish when $N\to\infty$.

	For $F_0^N(t)$, since we have
	\[|F_0^N(t)|
	\le 2N^{-2}\int_0^T\|\nabla^Nk^{N,\delta}(t)\|_{L^2(D)^d}^2\,dt+
	2N^{-2}\int_0^T\|\nabla^N\bar{h}^{N,\delta}(t)\|_{L^2(D)^d}^2\,dt,
	\]
	by the Schwarz inequality, we obtain
	\[\lim_{N\to\infty}\int_0^T|F_0^N(t)|\,dt=0\]
	from Proposition~\ref{prop2.2}.

	We shall next establish the bound for $F_2^N$ and $F^N_3$.
	To do so, we shall at first make an $L^2$ bound for $\nabla^N \bar{h}^N$
	on $B_N=\overline{D_N}\smallsetminus D_N$.
	Choosing $r,q>1$ such that $r=p/2,1/r+1/q=1$
	and applying the H\"older inequality, we obtain
	\begin{align}
	N^{-d}&\sum_{x\in\overline{B_N}}\left|\nabla^N
	\bar{h}^{N,\delta}(t,x/N)\right|^2 \nonumber \\
	&\le
	\frac{1}{r}\Bigl(N^{-d}\left|\overline{B_N}\right|\Bigr)^{r/2q}N^{-d}\sum_{x\in\overline{D_N}}\left|\nabla^N
	\bar{h}^{N,\delta}(t,x/N)\right|^p 
	+\frac{1}{q}\Bigl(N^{-d}\left|\overline{B_N}\right|\Bigr)^{1/2}.
	\label{eq2.23}
	\end{align}
	For $F^N_2(t)$, since we have
	\[|F_2^N(t)|\le N^{-d}\sum_{x\in\overline{B_N}}\left|\nabla^N
	\bar{h}^{N,\delta}(t,x/N)\right|^2,\]
	we obtain
	\[\lim_{N\to\infty}\int_0^T|F_2^N(t)|\,dt=0\]	
	from Proposition~\ref{prop2.8}.
	Also for $F^N_3(t)$, we have
	\begin{align*}
	|F_3^N(t)|\le & 2\alpha(N)N^{-d}\sum_{x\in\overline{B_N}}\left|\nabla^N
	\bar{h}^{N,\delta}(t,x/N)\right|^2 \\
	&\quad {}+2\alpha(N)^{-1}N^{-d}\sum_{x\in\overline{D_N}}\left|\nabla^N
	k^{N,\delta}(t,x/N)\right|^2
	\end{align*}
	for an arbitrary sequence $\{\alpha(N)\}$ of positive numbers.
	Choosing $\alpha(N)=N^{\epsilon}$ with $\epsilon>0$ small enough,
	we obtain
	\[\lim_{N\to\infty}\int_0^T|F_3^N(t)|\,dt=0.\]	
	from Proposition~\ref{prop2.8}.
	Summarizing above, we conclude \eqref{eq2.23b}.
\end{proof}
	
\subsection{Convergence of the solution for the discretized equation}
We shall show that the solution $\bar{h}^{N,\delta}$
of the discretized equation \eqref{discretized_PDE}
converges to the solution $h^\delta$
for the regularized equation \eqref{PDE1-reg}, when
the initial data is smooth enough. The goal is the following:
\begin{theorem}\label{th-conv-discrete-pde}
		Assume $h_0\in C_0^\infty(D)$. Then,
		the sequence of solutions $\{\bar{h}^{N,\delta}\}$
		for the discretized equation \eqref{discretized_PDE}
		with initial datum $h_0^N$
		converges as $N\to\infty$ to the unique solution $h^\delta(t)$
		of \eqref{PDE1-reg} with initial data $h_0$
		in the following sense:
		\[\lim_{N\to\infty}\|\bar{h}^{N,\delta}(t)-h^\delta(t)\|_{H^{1}(D)^*}=0\]
		holds for every $t>0$.
\end{theorem}

\begin{proof}
	To simplify notations, we omit the parameter $\delta$ when no confusion
	arises.
	We shall at first show that we can take a subsequence $\{N'\}$ such that
	$\bar{h}^{N'}$ converges to the solution of \eqref{PDE1-reg}.
	We arbitrarily choose $f\in C_0^\infty(D)$ such that
	\[\int_D h_0(\theta)\,d\theta=\int_D f(\theta)\,d\theta\]
	and define $f^N$ by
	\[f^{N}(\theta)=N^d\int_{B(x/N,1/N)}f(\theta')\,d\theta',
	\quad \theta\in B(x/N,1/N),\,x\in\integer^d.\]
	We introduce the polilinear interpolation used in \cite{DGI00},
	that is, $\hat{h}^{N}$ is defined by follows:
	\begin{equation}\label{polilinear}
	\hat{h}^{N}(t,\theta)
	=\sum_{\alpha\in\{0,1\}^d}\left[\prod_{i=1}^d\left(
	\alpha_i\{N\theta_i\}+(1-\alpha_i)(1-\{N\theta_i\}\right)\right]
	\bar{h}^N\left(t,\frac{[N\theta]+\alpha}{N}\right),
	\end{equation}
	where $[\cdot]$ and $\{\cdot\}$ denote the integral and the fractional parts,
respectively. We also define $\hat{k}^N$ by the similar manner.
	We then have
	\[\sup_{N}\sup_{0\le t\le T}\|\hat{h}^{N}(t)\|_{H^1(D)}<\infty\]
	and
	\[\sup_{N}\sup_{0\le t\le T}\|\hat{k}^{N}(t)\|_{H^1(D)}<\infty\]
	by Proposition~\ref{prop2.2} and Corollary~\ref{cor2.7}.
	Using Proposition~\ref{prop2.2}, Corollary~\ref{cor2.4}, Corollary~\ref{cor2.7} and the bounds stated above, we can choose
	a subsequence $\{N'\}$ such that
	\begin{align*}
		\bar{h}^{N'}-f^N\to\bar{g}&\quad \text{strongly in $C([0,T],H)$,} \\
		\hat{h}^N\to\hat{h}&\quad \text{weakly in $L^2([0,T],H^1_0(D))$,} \\
		\bar{k}^N\to\bar{k}&\quad \text{weakly in $L^2([0,T]\times D)$,} \\
		\hat{k}^N\to\hat{k}&\quad \text{weakly in $L^2([0,T],H^1(D))$} 
	\end{align*}
	as $N'\to\infty$ for some $\bar{g},\hat{h},\bar{k},\hat{k}$.
	Letting $\bar{h}=\bar{g}+f$,
	we can easily see that $\bar{h}=\hat{h}$ and $\bar{k}=\hat{k}$.
	Furthermore, in this setting, for every $t>0$,
	$\|\bar{h}^{N'}(t)-f^N\|_{-1,N'}$
	converges to $\|\bar{g}(t)\|_H$ as $N'\to\infty$, see Proposition~\ref{prop-d1}.
	Applying the argument in Step~3 of Proposition~I.2 in \cite{FS97},
	we obtain
	that the limit $\bar{h}$ is the solution of \eqref{PDE1-reg}
	with initial data $\bar{h}_0$.
	Furthermore, the uniqueness for \eqref{PDE1-reg}
	implies that the sequence $\{\bar{h}^{N};N\ge 1\}$ itself
	converges to $\bar{h}$ strongly in $C([0,T],H^{1}(D)^*)$,
	which shows the conclusion.
\end{proof}

\section{Identification of equilibrium states}\label{sec-stationary}
In this section, let us study the structure of
the equilibrium states for the dynamics on
$(\integer^d)^*$ corresponding to \eqref{SDE1}.
We will focus our attention to the relationship
between stationarity and Gibbs property, since we have already
known that the family of extremal canonical Gibbs measures
coincides with the family of extremal grand canonical Gibbs measures
introduced by \cite{FS97}, see \cite{N02a} for details.

\subsection{Notations}
In order to characterize the equilibrium states, we shall prepare several notations precisely.
Note that we will follow the same manner as in \cite{FS97} and \cite{N03}.

Let $(\integer^d)^*$ be the set of all directed bonds $b=(x,y),\,
x,y\in\integer^d,|x-y|=1$ in $\integer^d$.
We write $x_b=x$ and $y_b=y$ for $b=(x,y)$.
We denote the bond $(e_i ,0)$ by $e_i $ again
if it doesn't cause any confusion.
For every subset $\Lambda$ of $\integer^d$, we denote
the set of all directed bonds included $\Lambda$ and touching $\Lambda$
by $\Lambda^*$ and $\overline{\Lambda^*}$, respectively.
That is,
\begin{align*}
	\Lambda^*&:=\{b\in(\integer^d)^*;\,x_b\in\Lambda
	\text{ and }y_b\in\Lambda\},\\
	\overline{\Lambda^*}&:=\{b\in(\integer^d)^*;\,x_b\in\Lambda
	\text{ or }y_b\in\Lambda\}.
\end{align*}

For $\phi=\{\phi(x);\,x\in\integer^d\}\in\real^{\integer^d}$,
the gradient $\nabla$ is defined by
\begin{gather*}
\nabla\phi(b):=\phi(x)-\phi(y),\quad
b=(x,y)\in(\integer^d)^*.
\end{gather*}
Now, let $\mathcal{X}$ be the family of
all gradient fields $\eta\in\real^{(\integer^d)^*}$ which satisfy the plaquette
condition (2.1) in \cite{FS97}, i.e., $\mathcal{X}=\{\eta\equiv\nabla\phi;\,\phi\in \real^{\integer^d}\}$.
Let $\mathbb{L}^2_r$ be the set of all $\eta\in\real^{(\integer^d)^*}$
such that
\[|\eta|_r^2:=\sum_{b\in(\integer^d)^*}|\eta(b)|^2e^{-2r|x_b|}<\infty.\]
We denote $\mathcal{X}_r=\mathcal{X}\cap\mathbb{L}^2_r$ equipped with the
norm $|\cdot|_r$.
%

In this section, we study the properties of stationary measures
for the dynamics $\eta_t\in \mathcal{X}$ governed by the SDEs
\begin{equation}\label{SDE2}
d\eta_t(b)=-\nabla U_\cdot(\eta_t)(b)\,dt+\sqrt{2}d\nabla \tilde{w}^{\integer^d}_t(b),\quad b\in
(\integer^d)^*,
\end{equation}
where $\{w^{\integer^d}_t(x);\,x\in\integer^d\}$ is the family of Gaussian
processes with mean zero and covariance structure
\[
	E[\tilde{w}^{\integer^d}_t(x)\tilde{w}^{\integer^d}_s(y)]=
	-\Delta_{\integer^d}(x,y)t\wedge s,\quad x,y\in\integer^d,t,s\ge0.
\]
Since the coefficients are Lipschitz continuous in $\mathcal{X}_r$,
this equation has the unique strong solution in $\mathcal{X}_r$
for every $r>0$.
Note that $\eta_t:=\nabla\phi_t$ defined from the solution
$\phi_t$ of the SDE \eqref{SDE1} on $D_N$ satisfies \eqref{SDE2} for $b\in \overline{D_N^*}$ and boundary conditions $\eta_t(b)=\nabla\psi^N(b)$ for $b\in(\integer^d)^*\smallsetminus \overline{D_N^*}$,
when replacing $\tilde{w}^{\integer^d}_t$ by $\tilde{w}_t$.

Let us introduce the infinitesimal generator associated with \eqref{SDE2}.
For $\Lambda\subset \integer^d$, we introduce the 
the differential operator $\mathscr{L}_{\Lambda}$
of second order by
%
\[
	\mathscr{L}_{\Lambda}=
	-4\sum_{x\in \Lambda}\partial_x(\Delta_{\Lambda}\partial_{\cdot})(x)
	+2\sum_{x\in \Lambda}(\Delta_{\Lambda} U_\cdot(\eta))(x)\partial_x,
\]
where $\partial_x$ is defined by
\[
\partial_x:=\sum_{b\in(\integer^d)^*;\,x_b=x}\frac{\partial}{\partial\eta(b)}
,\quad x\in \integer^d.
\]
and $\Delta_\Lambda$ is the Laplacian on $\Lambda$ defined by
\eqref{laplacian-gamma} with $\Gamma=\Lambda$.
We note that the operator $\mathscr{L}_{\integer^d}$
is the infinitesimal generator associated
with the dynamics $\eta_t$ defined by \eqref{SDE2}.
To make notations keep simple, we simply denote $\mathscr{L}_{\integer^d}$
by $\mathscr{L}$ if it does not cause any confusion.
We also note that
$\mathscr{L}_{D_N}$ is the infinitesimal generator associated
with $\eta_t=\nabla\phi_t$, where $\phi_t$ is the solution of \eqref{SDE1}.

\subsection{Gibbs measures}
In this subsection, we state the definition of Gibbs measures with details.
For a finite set $\Lambda\subset\integer^d$ and fixed $\xi\in\mathcal{X}$, we define
the affine space $\mathcal{X}_{\Lambda,\xi}\subset\mathcal{X}$
by
\[\mathcal{X}_{\overline{\Lambda^*},\xi}=\{\eta\in\mathcal{X};\, \eta(b)=\xi(b),\,b\in(\integer^d)^*
\smallsetminus \overline{\Lambda^*}\}.\]
We define the finite volume Gibbs measure $\mu_{\Lambda,\xi}$
on $\overline{\Lambda^*}$ by
\[
\mu_{\Lambda,\xi}(d\eta)=Z^{-1}_{\Lambda,\xi}\exp\left(
-\sum_{b\in\overline{\Lambda^*}}V(\eta(b))\right)d\eta_{\overline{\Lambda^*},\xi},
\]
where $d\eta_{\overline{\Lambda^*},\xi}$ is the Lebesgue measure on $\mathcal{X}_{\overline{\Lambda^*},\xi}$
and $Z_{\Lambda,\xi}$ is the normalizing constant.

Let $\mathcal{P}(\mathcal{X})$ be the set of all probability measures on
$\mathcal{X}$ and let $\mathcal{P}_2(\mathcal{X})$ be those
$\mu\in\mathcal{P}(\mathcal{X})$ satisfying $E^\mu[|\eta(b)|^2]<\infty$
for each $b\in(\integer^d)^*$. 
The measure $\mu\in\mathcal{P}_2(\mathcal{X})$ is sometimes
called tempered. 
Let $\mathcal{G}$ be the family of translation invariant, tempered Gibbs
measures $\mu\in\mathcal{P}_2(\mathcal{X})$
introduced by \cite{FS97}, 
namely, the family of translation invariant
$\mu\in\mathcal{P}_2(\mathcal{X})$ satisfying the
Dobrushin-Lanford-Ruelle equation
\begin{equation}\label{DLR-eq}
\mu(\cdot|\mathscr{F}_{(\integer^d)^*\smallsetminus\overline{\Lambda^*}})=\mu_{\Lambda,\xi}(\cdot),\quad \text{$\mu$-a.s. $\xi$},
\end{equation}
where $\mathscr{F}_{(\integer^d)^*\smallsetminus\overline{\Lambda^*}}$
is the $\sigma$-algebra generated by $\left\{\eta(b);\,b\in (\integer^d)^*\smallsetminus\overline{\Lambda^*}\right\}$.
Note that the dynamics $\eta_t$ given by \eqref{SDE2} is reversible under
$\mu\in\mathcal{G}$.
We denote the family of $\mu\in\mathcal{G}$ with ergodicity under spatial shifts
by $\mathcal{G}_{\mathrm{ext}}$. Properties of Gibbs measures
are studied quite well, see \cite{FS97} and \cite{DGI00} for details.

\subsection{Relationship between stationary measures and Gibbs measures}
It is not difficult to show that every Gibbs measure $\mu\in\mathcal{G}$
is stationary with respect to $\mathscr{L}$.
Let us show that the converse is also true 
and therefore the stationarity is equivalent to the Gibbs property.
\begin{theorem}\label{th_stationary}
	If $\mu\in\mathcal{P}_2(\mathcal{X})$ is translation invariant
	and satisfies
	\[\int_{\mathcal{X}}\mathscr{L}F(\eta)\,d\mu=0\]
	for every $F\in C^2_{b,\mathrm{loc}}(\mathcal{X})$, then
	$\mu\in\mathcal{G}$.
\end{theorem}

The proof of Theorem~\ref{th_stationary} is similar to \cite{DNV11},
which is based on \cite{F82a}.
We at first introduce $\Phi_{\lambda}:\real\to\real$ by
\[\Phi_{\lambda}(u)=\frac{\lambda}{a}\left(1+(\lambda u)^2\right)^{-m},\]
where
\[a=\int_{\real}(1+u^2)^{-m}du.\]
For $\Lambda_{n}:=[-n,n]^d\cap\integer^d$
we define $\Phi^{\lambda}_{n}:\mathcal{X}_{\Lambda_n^*}\to\real$ by
\[\Phi^{\lambda}_{n}(\eta)=\prod_{x\in\Lambda_n}\Phi_{\lambda}(\phi^{\eta,0}(x)),\]
where $\phi^{\eta,c}$ is the height variable satisfying $\nabla\phi^{\eta,c}=\eta$ and
$\phi^{\eta,c}(0)=c$. Note that $\phi^{\eta,c}$ is uniquely determined by
$\eta$ and $c$.
We also define $p^\lambda_n(\eta)$ by
\[p^\lambda_n(\eta)=\int \Phi^{\lambda}_n(\eta-\xi)\mu(d\xi).\]
Let $\Psi^\lambda_n(\eta,\xi)=\Phi^{\lambda}_{n}(\xi-\eta)$. Since $\Psi^\lambda_n(\cdot,\xi)\in C^2_{\mathrm{loc}}(\mathcal{X})$, we have
\[\int \mathscr{L}\Psi^\lambda_n(\cdot,\xi)(\eta)\mu(d\eta)=0.\]
Multiplying 
$F(\xi)\in C^2_{b,\mathrm{loc}}(\mathcal{X})$ whose support is in $\Lambda_n^{*}$, and
integrating in $\xi$ by the uniform measure $d\xi_{\Lambda_n}$ on
\[
	\mathcal{X}_{\Lambda_n}=\{\nabla\phi\in\real^{\Lambda_n^*};
	\,\phi\in\real^{\Lambda_n}\},
\]
we obtain
\begin{equation}\label{eq3.1}
	\iint F(\xi)\mathscr{L}\Psi^\lambda_n(\cdot,\xi)(\eta)
	\mu(d\eta)d\xi_{\Lambda_n}=0.
\end{equation}

Here, as in Lemma~3.2 in \cite{DNV11}, if $f:\real^{\integer^d}\to\real$ is smooth, local
and the form 
$f(\phi)=F(\nabla\phi)$ for some $F:\mathcal{X}\to\real$, we then have
\[
	\frac{\partial f}{\partial \phi(x)}=2\sum_{b:x_b=x}\frac{\partial F}{\partial \eta(b)}(\nabla\phi).
\]
Noting the above and the relationship
\begin{gather*}
	\frac{\partial\Psi^\lambda_n(\nabla\cdot,\nabla\psi)}{\partial\phi(x)}(\phi)
	=-\frac{\partial \Phi^{\lambda}_n}{\partial\phi(x)}(\nabla\psi-\nabla\phi)=
	-\frac{\partial\Psi^\lambda_n(\nabla\phi,\nabla\cdot)}{\partial\psi(x)}(\psi) \\
	\frac{\partial^2\Psi^\lambda_n(\nabla\cdot,\nabla\psi)}{\partial\phi(x)^2}(\phi)
	=\frac{\partial^2 \Phi^{\lambda}_n}{\partial\phi(x)^2}(\nabla\psi-\nabla\phi)
	=\frac{\partial^2\Psi^\lambda_n(\nabla\phi,\nabla\cdot)}{\partial\psi(x)^2}(\psi)
\end{gather*}
for $x\in\integer^d$ by the symmetricity of $\Phi^{\lambda}$,
the left hand side of \eqref{eq3.1} is calculated as follows:
\begin{align*}
	&\iint F(\nabla\psi)\sum_{x\in \Lambda_n}\sum_{y\in \Lambda_n}
	(-\Delta)(x,y)\frac{\partial^2\Psi_n^\lambda(\eta,\nabla\cdot)}{\partial\psi(x)\partial\psi(y)}\nu_{\Lambda_n,p}(d\psi)\mu(d\eta) \\
	&\quad {}+\iint F(\nabla\psi)\sum_{x\in \Lambda_n}
	\sum_{y\in \integer^d}(-\Delta)(x,y)U_y(\eta)
	\frac{\partial\Psi_n^\lambda(\eta,\nabla\cdot)}{\partial\psi(x)}
	\nu_{\Lambda_n,p}(d\psi)\mu(d\eta) \\
	&=:I_1+I_2,
\end{align*}
where $\nu_{\Lambda_n,p}$ is the measure on $\real^{\Lambda_n}$ defined by
\[\nu_{\Lambda_n,p}(d\psi)=p(\psi(0))\prod_{x\in\Lambda_n}d\psi(x)\]
with a probability density $p$ on $\real$. We have used that
the image measure  of $\nu_{\Lambda_n,p}$ by
the discrete gradient $\nabla$ coincides with $d\xi_{\Lambda_n}$.
Performing the integration-by-parts for $I_1$, we have
We shall first  calculate $I_1$. Performing integration-by-parts in $\psi$, we have
\begin{align}
	I_1&=-\iint \sum_{x\in \Lambda_n}\sum_{y\in \Lambda_n}(-\Delta)(x,y)
	\frac{\partial F(\nabla\cdot)}{\partial \psi(x)} 
	\frac{\partial \Psi^\lambda_n(\eta,\nabla\cdot)}{\partial\psi(y)}\nu_{\Lambda_n,p}(d\psi)\mu(d\eta)  \nonumber \\
	&\qquad {}-\iint \sum_{y\in\Lambda_n}(-\Delta)(x,y)
	F(\nabla\psi) 
	\frac{\partial \Psi^\lambda_n(\eta,\nabla\cdot)}{\partial\psi(0)}p'(\psi(0))
	\prod_{x\in\Lambda_n}d\psi(x)\mu(d\eta) . \label{eq3.2}
\end{align}
Noting that integrands of $I_1$ and the first term in the right hand side of \eqref{eq3.2} are
function of $\nabla\psi$, each integral does not depend on the choice of $p$
and therefore the second term does not also. Taking a sequence $p_n$ such that $p_n'\to0$ as $n\to\infty$,
we conclude that the second term must be zero.

Let us choose $F$ as
\begin{equation}\label{eq3.4}
	F(\nabla\psi)=f\left(\frac{p^\lambda_n(\nabla\psi)}{q_n(\nabla\psi)}\right),
\end{equation}
with some bounded smooth function $f:\real\to\real$ and
\[q_n(\eta)=\exp\left(-H^\Lambda(\eta)\right),\quad \eta\in\mathcal{X}_{\Lambda_n^*},\]
where $Z_n$ is the normalizing constant.
Noting
\[\frac{\partial p^\lambda_n(\nabla\cdot)}{\partial \psi(x)}=\int\frac{\partial \Psi^\lambda_n(\eta,\nabla\cdot)}{\partial\psi(x)}\mu(d\eta)\]
and putting
\[
	r_n^\lambda(\eta)=\frac{p^\lambda_n(\eta)}{q_n(\eta)},
\]
we have
\begin{align}
I_1
&=-\sum_{x\in\Lambda_n}\sum_{y\in\Lambda_n}(-\Delta)(x,y)\int  f'\left(r^\lambda_n(\nabla\psi)\right)
\frac{\partial r_n^\lambda(\nabla\cdot)}{\partial \psi(x)}
\frac{\partial r_n^\lambda(\nabla\cdot)}{\partial \psi(y)}
q_n(\nabla\psi)\nu_{\Lambda_n,p}(d\psi)  \nonumber \\
	&\qquad {}+\sum_{x\in\Lambda_n}\sum_{y\in\Lambda_n}(-\Delta)(x,y)
	\int f'\left(r^\lambda_n(\nabla\psi)\right)
U^\Lambda_{y}(\nabla\psi)
	p^\lambda_n(\nabla\psi)
	\nu_{\Lambda_n,p}(d\psi). \label{eq2.2c}
\end{align}
Next, we shall compute $I_2$. Performing the integration-by-parts in $\psi(x)$ again, we have
\begin{align}
	I_2
	&=\sum_{x\in \Lambda_n}
	\sum_{y\in \integer^d\smallsetminus \Lambda_n}(-\Delta)(x,y)\iint 
	\frac{\partial F(\nabla\psi)}{\partial\psi(x)}
	U_y(\eta)\Psi^\lambda_n(\eta,\nabla\psi)
	\nu_{\Lambda_n,p}(d\psi)\mu(d\eta) \nonumber \\
	& \qquad {}-\sum_{x\in\Lambda_n}\sum_{y\in\Lambda_n}(-\Delta)(x,y)
	\iint \frac{\partial F(\nabla\psi)}{\partial\psi(x)}
	\left(U_y(\eta)-U^\Lambda_y(\eta)\right) \nonumber \\
	&\qquad \qquad \qquad \qquad \qquad \qquad \qquad \qquad \times
	\Psi^\lambda_n(\eta,\nabla\psi)\nu_{\Lambda_n,p}(d\psi)\mu(d\eta) \nonumber 
	\\
	& \qquad {}-\sum_{x\in\Lambda_n}\sum_{y\in\Lambda_n}(-\Delta)(x,y)
	\iint \frac{\partial F(\nabla\psi)}{\partial\psi(x)}
	\left(U_y^\Lambda(\eta)-U_y^\Lambda(\nabla\psi)\right) \nonumber \\
	&\qquad \qquad \qquad \qquad \qquad \qquad \qquad \qquad \times
	\Psi^\lambda_n(\eta,\nabla\psi)\nu_{\Lambda_n,p}(d\psi)\mu(d\eta) \nonumber 
	\\
	& \qquad {}-\sum_{x\in\Lambda_n}\sum_{y\in\Lambda_n}(-\Delta)(x,y)
	\iint \frac{\partial F(\nabla\psi)}{\partial\psi(x)}
	U_y^\Lambda(\nabla\psi)
	p^\lambda_n(\nabla\psi)\nu_{\Lambda_n,p}(d\psi) \nonumber \\
	&=:R_1^\lambda(n,f)+R_2^\lambda(n,f)+R_3^\lambda(n,f)+R_4^\lambda(n,f).
	\label{eq3.5}
\end{align}
Summarizing \eqref{eq3.2} and \eqref{eq3.5}, we obtain
\begin{align}
	F^\lambda(n,f)&:=\sum_{x\in\Lambda_n}\sum_{y\in\Lambda_n}(-\Delta)(x,y)
	\int  f'\left(r^\lambda_n(\nabla\psi)\right)
\frac{\partial r_n^\lambda(\nabla\cdot)}{\partial \psi(x)}
\frac{\partial r_n^\lambda(\nabla\cdot)}{\partial \psi(y)} \\
	&\qquad \qquad \qquad \qquad \qquad \qquad \qquad \qquad \times
q_n(\nabla\psi)\nu_{\Lambda_n,p}(d\psi) \nonumber \\
	&=R_1^\lambda(n,f)+R_2^\lambda(n,f)+R_3^\lambda(n,f)
	\label{eq3.5c}
\end{align}
if we take $F$ as in \eqref{eq3.4}. 
Note that $F^\lambda(n,f)$ is finite when $f(x)=\log x$
by using $|\Delta(x,y)|\le 2d$, see Lemma~3.3 of \cite{DNV11}.
We denote $F^\lambda(n,f)$ with $f(x)=\log x$ simply by $F^\lambda(n)$.

From now on, we shall show that terms $R_i^\lambda(n,f)$ 
in the right hand side can be controlled
by $F^\lambda(n)$. 
\begin{lemma}
	Assume that the function $f$ satisfies $0\le uf'(u)\le 1$ for every $u>0$.
	We then have bounds for $R^\lambda_{1}(n,f), R^\lambda_{2}(n,f)$
	and $R^\lambda_{3}(n,f)$ in \eqref{eq3.5} as follows:
	\begin{gather}
		\left|R^\lambda_{1}(n,f)\right|\le K C(n)^{1/2}F^\lambda(n)^{1/2} \label{eq3.6}\\
		\left|R^\lambda_{2}(n,f)\right|\le K \lambda^{-1}F^\lambda(n)^{1/2} \label{eq3.7} \\
		\left|R^\lambda_{3}(n,f)\right|\le K C(n)^{1/2}F^\lambda(n)^{1/2} \label{eq3.8}
	\end{gather}
	with a constant $K>0$ independent in $n$ and $\lambda$, where $C_x(n)$ is
	defined by
	\begin{gather*}
		C(n)=\sum_{x\in\Lambda_{n+1}}\sum_{b\in(\integer^d)^*\smallsetminus\Lambda^*:x_b=x}\int c^2_b(\xi,n,\mu)
		\mu(d\xi), \\
		c_b(\xi,n,\mu)=\int V'(\eta(b))\mu(d\eta|\mathscr{F}_{(\Lambda_n^*)^\complement})(\xi).
	\end{gather*}
	Here, $\mathscr{F}_{(\Lambda_n^*)^\complement}$ is $\sigma$-algebra
	generated by $\{\eta(b);b\in (\Lambda_n^*)^\complement\}$.
\end{lemma}
\begin{proof}
	We at first obtain 
	\begin{align*}
	\iint &f'\left(\frac{p^\lambda_n}{q_n}\right)^2
	\sum_{x\in\Lambda_n}
	\sum_{y\in\Lambda_n}
	(-\Delta)(x,y)
	\frac{\partial r_n^\lambda}{\partial\psi(x)}
	\frac{\partial r_n^\lambda}{\partial\psi(y)}
	\Psi^\lambda_n(\eta,\nabla\psi)\nu_{\Lambda_n,p}(d\psi)\mu(d\eta) \\
	&\le \iint \left(\frac{p^\lambda_n}{q_n}\right)^{-2}
		\sum_{x\in\Lambda_n}
	\sum_{y\in\Lambda_n}
	(-\Delta)(x,y)
	\frac{\partial r_n^\lambda}{\partial\psi(x)}
	\frac{\partial r_n^\lambda}{\partial\psi(y)}
p^\lambda_n(\nabla\psi)\nu_{\Lambda_n,p}(d\psi) \\
&= \int \left(\frac{p^\lambda_n}{q_n}\right)^{-1}
	\sum_{x\in\Lambda_n}
	\sum_{y\in\Lambda_n}
	(-\Delta)(x,y)
	\frac{\partial r_n^\lambda}{\partial\psi(x)}
	\frac{\partial r_n^\lambda}{\partial\psi(y)}
q_n(\nabla\psi)\nu_{\Lambda_n,p}(d\psi)=F^\lambda(n)
\end{align*}
by assumption on $f$. Note that $(-\Delta)$ is nonnegative
definite, indeed, we have
\[
	\sum_{x\in\Lambda_m}(-\Delta)(x,y)\phi(x)\phi(y)\ge 0
\]
for every $m\ge1$ and $\phi\in\real^{\integer^d}$ such that $\phi(x)=0$ on $\Lambda_m^\complement$.
Furthermore, we have the Schwarz inequality
of the following form:
\begin{align}
	&\left|\sum_{x\in\integer^d}
	\sum_{y\in\integer^d}(-\Delta)(x,y)\phi(x)\psi(y)\right| \nonumber \\
	&\quad \le
	\left(\sum_{x\in\integer^d}
	\sum_{y\in\integer^d}(-\Delta)(x,y)\phi(x)\phi(y)\right)^{1/2}
	\left(\sum_{x\in\integer^d}
	\sum_{y\in\integer^d}(-\Delta)(x,y)\psi(x)\psi(y)\right)^{1/2}
	\label{schwarz_delta}
\end{align}
for every $\phi,\psi\in\real^{\integer^d}$ with
$\phi(x)=\psi(x)=0$ on $\Lambda_m^\complement$ for some $m\ge1$.
Using the above, we obtain
\begin{align*}
\left|R^\lambda_{1}(n,f)\right|
&\le C F^\lambda(n)^{1/2}C(n)^{1/2}
\end{align*}
for some constant $C>0$, which shows \eqref{eq3.6}.
We can also obtain \eqref{eq3.7} and \eqref{eq3.8} by the similar argument
to the above. 
\end{proof}

Let us continue the calculation for $F^\lambda(n,f)$.
Summarizing \eqref{eq3.5c}-\eqref{eq3.8}, we get
\[
F^\lambda(n,f)\le \frac{1}{2}F^\lambda(n)+
 K'(C(n)+\lambda^{-1}n^d)
\]
with a constant $K'>0$. Using Fatou's lemma, we conclude
\begin{equation}\label{eq3.11}
\frac{1}{2}F^\lambda(n)\le K'(C(n)+\lambda^{-1}n^d).
\end{equation}
We shall next give a lower bound for the left hand side of \eqref{eq3.11}. 
For $\ell\in\natural$, let us take $\tilde\Lambda\subset\Lambda_n$ by
\[\tilde\Lambda=\bigcup_{x\in\mathscr{A}_{n,\ell}}\Lambda_\ell(x),\]
where $\Lambda_\ell(x)=\Lambda_\ell+x$ and
\[
\mathscr{A}_{n,\ell}=\{x\in((2\ell+3)\integer)^d;\Lambda_\ell(x)\subset\Lambda_{n-1}\}.
\]
Because boxes $\Lambda_\ell(x)$ appearing above are disjoint, we get
\begin{align*}
F^\lambda(n)&=\frac{1}{2}\int\sum_{x\in\Lambda_n}\sum_{y\in\integer^d;|x-y|=1}
\left(\frac{\partial \sqrt{r_n^\lambda}}{\partial\psi(x)}
-\frac{\partial \sqrt{r_n^\lambda}}{\partial\psi(y)}\right)^2
q_n(\nabla\psi)\nu_{\Lambda_n,p}(d\psi) \\
&\ge\frac{1}{2}\int\sum_{z\in\mathscr{A}_{n,\ell}}
\sum_{x\in \Lambda_\ell(z)}
\sum_{y\in\Lambda_\ell(z);|x-y|=1}
\left(\frac{\partial \sqrt{r_n^\lambda}}{\partial\psi(x)}
-\frac{\partial \sqrt{r_n^\lambda}}{\partial\psi(y)}\right)^2 \\
&\qquad \qquad \qquad \qquad \qquad \qquad \qquad \qquad \times
q_n(\nabla\psi)\nu_{\Lambda_n,p}(d\psi) \\
&=\frac{1}{2}\sum_{z\in\mathscr{A}_{n,\ell}}I^{\Lambda_\ell(x)}(\mu_n^\lambda),
\end{align*}
where $I^\Lambda$ is the entropy production rate defined by
\[
	I^{\Lambda}(\tilde\mu)
	:=\sup\left\{\int\frac{-\mathscr{L}^{\Lambda}u}{u}d\tilde\mu;\,
	u\in C_{b}^2(\mathcal{X}),\,\mathscr{F}_{\overline{\Lambda^*}}\text{-measurable},\,u\ge 1\right\}
\]
for a finite $\Lambda\subset\integer^d$ and $\tilde\mu$ on $\mathcal{X}$
or $\mathcal{X}_{\Lambda_n}$ with $n$ large enough.
Applying \eqref{eq3.11} and taking the limit $\lambda\to\infty$, we have
\[
\frac{1}{2}\sum_{z\in\mathscr{A}_{n,\ell}}I^{\Lambda_\ell(x)}(\mu)
\le 4K'C(n)
\]
by the lower semicontinuity of the entropy production rate.
Since $I^{\Lambda_\ell(x)}(\mu_n)$ does not depend on $x$ by the
translation invariance of $\mu$ and $C(n)=O(n^{d-1})$, we get
\begin{equation}\label{eq3.12}
I^{\Lambda_\ell}(\mu)\le C\ell^dn^{-1}
\end{equation}
with a constant $C>0$ independent of $n$.
Taking the limit $n\to\infty$, 
we finally conclude that
\begin{equation}\label{entropy-bound2}
I^{\Lambda_\ell}(\mu)=0
\end{equation}
for every $\ell\in\natural$. Repeating the same argument as in
the proof of Lemma~5.2 in \cite{N02a}, we obtain
that $\mu$ is a canonical Gibbs measure introduced in \cite{N02a}.
Applying Theorem~3.1 in \cite{N02a}, we conclude $\mu\in \mathcal{G}$.

\section{Proof of Theorem~\ref{hydro}}\label{sec-pf-mainthm}
In this section, we shall give the proof of our main result, Theorem~\ref{hydro}. Before that, we shall prepare several bounds, which play key role in the proof.
\subsection{A priori bounds for stochastic processes}
In this subsection, let us establish $L^2$-bound for the stochastic
interface $h^N$. 
\begin{prop}\label{apriori-SDE}
	There exist constants $C_1,C_2>0$ independent of $N$
	such that
	\[E\|h^N(t)\|^2_{-1,N}+N^{-d}E\int_0^t\sum_{x\in\overline{D_N^*}}\left(\nabla\phi^N_s(x)\right)^2\,ds
	\le C_1E\|h^N(0)\|^2_{-1,N}+C_2(1+t)\]
	holds for every $t\ge0$.
\end{prop}
\begin{proof}
	Let us use the function $g^N$ and $\zeta^N$
	introduced at Section~\ref{subsec-apriori-pde}.
	We define $\psi^N$ by
	\begin{equation*}
	\psi^N(x)=\langle \phi_0^N\rangle\zeta^N(x),\quad x\in\integer^d
	\end{equation*}
	where $\langle \phi_0^N\rangle$ is defined by
	\[\langle \phi_0^N\rangle=N^{-d-1}\sum_{y\in D_N}\phi_0^N(y).\]
	We denote the macroscopic height variable associated with the
	microscopic height variable $\psi^N$ by $f^N$, that is,
	\[f^N(\theta)=\sum_{x\in\integer^d}N^{-1}\psi^N(x)1_{B(x/N,1/N)}(\theta),
	\quad\theta\in\real^d.\]
	Calculating $\|h^N(t)-f^N\|_{-1,N}^2$ by It\^o's formula, we obtain
	{\allowdisplaybreaks\begin{align*}
	d\|h^N(t)-f^N\|^2_{-1,N} 
	&=-2N^{-d}\sum_{b\in \overline{D_N^*}}
	\left(\nabla\phi^N_t(b)-\nabla\psi^N(b)\right)V'(\nabla\phi_t^N(b))\,dt \\
	&\quad {}+2N^{-d}\sum_{x\in D_N}(-\Delta_{D_N})^{-1}
	\left(\phi^N_t-\psi^N\right)(x)\,d\tilde{w}_t(y) \\	
	&\quad {}+2N^{-d}|D_N|\,dt.
	\end{align*}}
	Therefore, integrating in $t$ and taking the expectation, we get
	{\allowdisplaybreaks\begin{align}\nonumber
	E\|h^N(T)-f^N\|^2_{-1,N}&=E\|h^N(0)-f^N\|^2_{-1,N} \\
	&\quad {}-2E\int_0^TN^{-d}\sum_{b\in \overline{D_N^*}}
	\nabla\phi^N_t(b)V'(\nabla\phi_t^N(b))\,dt \nonumber \\
	&\quad {}+2E\int_0^TN^{-d}\sum_{b\in \overline{D_N^*}}
	\nabla\psi^N(b)V'(\nabla\phi_t^N(b))\,dt \nonumber \\
	&\quad {}+2N^{-d}|D_N|T. \label{eq3.3}
	\end{align}}
	Applying \eqref{bound_aux} to the third term in the right hand side, we get
	\begin{align*}
	\left|2N^{-d}\sum_{b\in \overline{D_N^*}}\nabla \psi^N(b)
	V'(\nabla\phi^N_t(b))\right|
	&\le 2c_f|\langle \phi_0^N\rangle|
	\left(N^{-d}\sum_{b\in \overline{D_N^*}}
	\left(V'(\nabla\phi^N_t(b))\right)^2\right)^{1/2} \\
	&\le 2c_f|\langle \phi_0^N\rangle|c_{+}\left(N^{-d}\sum_{b\in \overline{D_N^*}}
	\left(\nabla\phi^N_t(b)\right)^2\right)^{1/2} \\
	&\le 2\gamma N^{-d}\sum_{b\in \overline{D_N^*}}
	\left(\nabla\phi^N_t(b)\right)^2 \\
	&\qquad {}+8\gamma^{-1}c_f^2|\langle \phi_0^N\rangle|^2c_{+}^2
	\end{align*}
	for every $\gamma>0$. Plugging the above with $\gamma=c_{-}/2$
	into \eqref{eq3.3}, we finally obtain
	\begin{align*}
	E\|h^N(T)-f^N\|^2_{-1,N}
	&\le E\|h^N(0)-f^N\|^2_{-1,N} \\
	&\quad {}-c_{-}E\int_0^TN^{-d}\sum_{b\in \overline{D_N^*}}
	\left(\nabla\phi^N_t(b)\right)^2\,dt \\
	&\quad {}+2N^{-d}|D_N|T+16c_{-}^{-1}c_{+}^2c_f^2|\langle \phi_0^N\rangle|^2T.
	\end{align*}
	Noting $|\langle \phi_0^N\rangle|^2
	\le \|h^N(0)\|_{-1,N}^2$, we obtain the conclusion.
\end{proof}

\subsection{Coupled local equilibria}
We shall introduce the coupled measure and
identify its limit point, as in \cite{FS97} and \cite{N03}.
Let us denote the discrete gradient of the solution $\bar{h}^{N,\delta}$
of \eqref{discretized_PDE} by $u_s^{N,\delta}$, that is,
\[
u_s^{N,\delta}(x)=\nabla^N\bar{h}^{N,\delta}(s,x/N),\quad x\in\integer^d,
\]
and the law of $\nabla\phi_s^N$ on $\mathcal{X}$ by $\mu_s^N$.
Using these notations, we introduce the coupled measure $p^N(d\eta\,du)$ on
$\mathcal{X}\times\real^d$ by
\begin{equation}\label{coupled}
p^{N,\delta}(d\eta\,du)=t^{-1}|\overline{D_N}|^{-1}
\int_0^t\sum_{x\in \overline{D_N}}(\mu_s^N\circ\tau_x^{-1})(d\eta)
\delta_{u_s^{N,\delta}(x)}(du)\,ds,
\end{equation}
where $\tau_x$ is the spatial shift by $x\in\integer^d$.
We note that the sequence $\{p^N\}$ is tight as the probability measures
on $\mathcal{X}\times\real^d$ since we have Proposition~\ref{apriori-SDE} and
Proposition~\ref{prop2.2}.

\begin{prop}
	For every limit point $p$ of $\{p^{N,\delta}\}$, there exists a probability measure
	$\lambda(du\,dv)$ on $\real^d\times\real^d$ such that
	the relationship
	\[p(d\eta\,du)=\int_{\real^d}\mu_v(d\eta)\lambda(du\,dv)\]
	holds with ergodic Gibbs measures $\{\mu_u;\,u\in\real^d\}$
	introduced by \cite{FS97}.
\end{prop}
\begin{proof}
	To keep notation simple, let us omit the parameter $\delta$ 
	when no confusion arises.
	 
	For $\varphi\in C_b^1(\real^d)$, we define the signed measure
	$p^N(d\eta,\varphi)$ on $\mathcal{X}$ by
	\[p^N(d\eta,\varphi)=\int_{\real^d}p^N(d\eta\,du)\varphi(u).\]
	It is sufficient to show that every limit point $p(d\eta,\varphi)$ of $\{p^N(d\eta,\varphi)\}$ is translation invariant
	and satisfies
	\begin{equation}\label{invariance}
	\int_\mathcal{X}\mathscr{L}F(\eta)p(d\eta,\varphi)=0
	\end{equation}
	for every $F\in C^2_b(\mathcal{X})$ with a compact support, see Theorem~4.1 of \cite{FS97}.
	
	We shall at first show the limit point $p(d\eta,\varphi)$ is translation invariant.
	For $F\in C^2_b(\mathcal{X})$ and $e\in\integer^d$ such that
	$|e|=1$, we have
	\begin{align*}
		&\left|\int_\mathcal{X}F(\eta)p^N(d\eta,\varphi)
		-\int_\mathcal{X}F(\eta)p^N(\cdot,\varphi)\circ\tau_e(d\eta)\right| \\
		&\quad \le
		t^{-1}|\overline{D_N}|^{-1}\|F\|_{\infty}\|\varphi'\|_{\infty}
\int_0^t\sum_{x\in \overline{\overline{D_N}}}
		\left|u^N(s,x/N)-u^N(s,x/N-e/N)\right|\,ds \\
		&\qquad {}+|\overline{D_N}|^{-1}\|F\|_{\infty}\|\varphi\|_{\infty}
		\left(\left|\overline{D_N}\cap (\overline{D_N}+e)^\complement\right|+
		\left|(\overline{D_N}+e)\cap \overline{D_N}^\complement\right|\right).
	\end{align*}
	Since the first term vanishes as $N\to\infty$ by Proposition~\ref{oscillation}, we obtain
	\[\lim_{N\to\infty}\left|\int_\mathcal{X}F(\eta)p^N(d\eta,\varphi)
		-\int_\mathcal{X}F(\eta)p^N(\cdot,\varphi)\circ\tau_e(d\eta)\right|=0,\]
	which shows that the limit $p(d\eta,\varphi)$ is translation invariant.
	
	Let us show \eqref{invariance}.
	Fix $F\in C^2_b(\mathcal{X})$ with compact support, and take
	$L\in\natural$ such that
	$\mathrm{supp}(F)\subset \Lambda^*_L$, where
	$\Lambda_L=[-L,L]^d\cap \integer^d$.
	We then obtain for $N$ large enough
	{\allowdisplaybreaks
	\begin{align*}
		\int_\mathcal{X}\mathscr{L}&F(\eta)p^N(d\eta,\varphi) \\
		&=t^{-1}|D_N|^{-1}\int_0^t\sum_{x\in \overline{D_N}}
		\varphi(u_s^N(x))E^{\mu_s^N\circ\tau_x}[\mathscr{L}F(\eta)]\,ds \\
		&=t^{-1}|D_N|^{-1}\int_0^t\sum_{x\in \mathscr{A}_{N,L}}
		\varphi(u_s^N(x))E^{\mu_s^N}[\mathscr{L}_{D_N}(F\circ\tau_x^{-1})(\eta)]\,ds \\
		&\qquad {}+t^{-1}|D_N|^{-1}\int_0^t
		\sum_{x\in \overline{D_N}\smallsetminus\mathscr{A}_{N,L}}
		\varphi(u_s^N(x))E^{\mu_s^N}[\mathscr{L}_{D_N}
		(F\circ\tau_x^{-1})(\eta)]\,ds \\
		&\qquad {}+t^{-1}|D_N|^{-1}\int_0^t\sum_{x\in \overline{D_N}}
		\varphi(u_s^N(x))E^{\mu_s^N\circ\tau_x}[\mathscr{L}_{\integer^d\smallsetminus D_N}F(\eta)]\,ds \\
		&=:F_1^N+F_2^N+F_3^N,
	\end{align*}}
	where the set $\mathscr{A}_{N,L}$ is defined by
	\[\mathscr{A}_{N,L}=\overline{D_N}\cap\left(\bigcup_{z\in\partial D_N}
	(\Lambda_{2L}+z)\right)^\complement.\]
	Since $\mathrm{supp}(F\circ\tau^{-1}_x)\subset \overline{D_N}$
	if $x\in \mathscr{A}_{N,L}$, we obtain
	\begin{align*}
		F_1^N&=t^{-1}|D_N|^{-1}N^{-4}\int_0^t\sum_{x\in \mathscr{A}_{N,L}}
		\varphi(u_s^N(x))E^{\mu_s^N}[(F\circ\tau_x^{-1})(\nabla\phi^N_t)] \\
		&\quad {}-t^{-1}|D_N|^{-1}N^{-4}\int_0^t\sum_{x\in \mathscr{A}_{N,L}}
		\varphi(u_s^N(x))E^{\mu_s^N}[(F\circ\tau_x^{-1})(\nabla\phi^N_0)] \\
		&\quad {}-t^{-1}|D_N|^{-1}N^{-4}\int_0^t\sum_{x\in \mathscr{A}_{N,L}}
		\frac{\partial}{\partial s}\varphi(u_s^N(x))
		E^{\mu_s^N}[(F\circ\tau_x^{-1})(\nabla\phi^N_s)]\,ds
	\end{align*}	
	by It\^o's formula. Since $F\in C^2_b(\mathcal{X})$, we obtain
	\begin{align*}
		|F_1^N|&\le C_1N^{-4}+C_2N^{-4}\int_0^t\left\|\frac{\partial}{\partial s}\varphi(u_s^N(x))\right\|_{L^1}\,ds
	\end{align*}	
	with some constants $C_1,C_2>0$ independent of $N$. Combining
	the above with Proposition~\ref{prop2.3}, we have
	\[\lim_{N\to\infty}|F_1^N|=0.\]

	For $F_2^N$, we have
	\begin{align*}
		|F_2^N|
		&\le t^{-1}\|\varphi\|_\infty |D_N|^{-1}\int_0^t
		\sum_{x\in \overline{D_N}\smallsetminus\mathscr{A}_{N,L}}
		E^{\mu_s^N}[\left|\mathscr{L}_{D_N}
		(F\circ\tau_x^{-1})(\eta)\right|]\,ds \\
		&\le t^{-1}\|\varphi\|_\infty |D_N|^{-1}\int_0^t
		\sum_{x\in \overline{D_N}\smallsetminus\mathscr{A}_{N,L}}
		\sum_{y\in \Lambda_{L}+x}
		E^{\mu_s^N}[\left|\mathscr{L}_{y}
		(F\circ\tau_x^{-1})(\eta)\right|]\,ds \\
		&\le C_3 |D_N|^{-1}
		\left|\overline{D_N}\smallsetminus\mathscr{A}_{N,L}\right| \\
		&\quad {}+ C_4|D_N|^{-1}\int_0^t
		\sum_{x\in \overline{D_N}\smallsetminus\mathscr{A}_{N,L}}
		\sum_{b\in(\integer^d)^*:x_b\in \Lambda_{2L}+x}
		E^{\mu_s^N}[\left|V'(\nabla\phi_s^N(b)\right|]\,ds \\
		&=:F_{2,1}^N+F_{2,2}^N
	\end{align*}
	with some constants $C_3,C_4>0$ independent of $N$.
	We can easily see that 
	\[\lim_{N\to\infty}F_{2,1}^N=0.\]
	For $F_{2,2}^N$, applying the Schwarz inequality, we obtain
	\begin{align*}
		F_{2,2}^N
		&\le  C_5N^{-1/2}|D_N|^{-1}E\int_0^t
		\sum_{b\in\overline{D_N^*}}\left|\nabla\phi_s^N(b)\right|^2\,ds 
		+C_6N^{1/2}|D_N|^{-1}|\partial D_N|
	\end{align*}
	with some constants $C_5,C_6>0$ independent of $N$, which indicates
	\[\lim_{N\to\infty}F^N_{2,2}=0\]
	from Proposition~\ref{apriori-SDE}. We therefore conclude
	\[\lim_{N\to\infty}F^N_{2}=0.\]
	Since we obtain
	\begin{align*}
		F_{3}^N
		&\le  C_7N^{-1/2}|D_N|^{-1}E\int_0^t
		\sum_{b\in\overline{D_N^*}}\left|\nabla\phi_s^N(b)\right|^2\,ds 
		+C_8N^{1/2}|D_N|^{-1}|\partial D_N|,
	\end{align*}
	with some constants $C_7,C_8>0$ independent of $N$
	by a similar argument to that for $F_2^N$, 
	we obtain
	\[\lim_{N\to\infty}F^N_{3}=0.\]
	Summarizing the above, the identity \eqref{invariance}
	is concluded.
\end{proof}
\subsection{Derivation of the macroscopic equation under \eqref{smoothness-initial}} \label{HDL-smooth-init}
Let us first prove Theorem~\ref{hydro} under the following assumption:
$h^N(0)$ is given by $h^N(0)=\bar{h}_0^N$, where $\bar{h}_0^N$ is
defined by \eqref{initial_PDE} with $h_0\in C_0^\infty(D)$.
Since we have
\begin{align*}
	\|h^N(t)-h(t)\|_{H^{1}(D)^*}
	&\le \|h^N(t)-\bar{h}^{N,\delta}(t)\|_{H^{1}(D)^*} \\
	&\qquad {}+\|\bar{h}^{N,\delta}(t)-h^{\delta}(t)\|_{H^{1}(D)^*} \\
	&\qquad {}+\|h^{\delta}(t)-h(t)\|_{H^{1}(D)^*},
\end{align*}
it is sufficient for our goal to show 
\begin{equation}\label{eq4.4}
	\limsup_{\delta\to0}\limsup_{N\to\infty}E\left\|h^N(t)-\bar{h}^{N,\delta}(t)\right\|_{-1,N}^2
	=0,
\end{equation}
by using Proposition~\ref{prop-sobolev-norms2}. 
Using It\^o's formula, we obtain
{\allowdisplaybreaks\begin{align*}
	E&\|h^N(t)-\bar{h}^{N,\delta}(t)\|^2_{-1,N} \\
	&=-tN^{-d}\Bigl|\overline{D_N}\Bigr|\sum_{i =1}^d
	\int_{\mathcal{X}\times\real^d}
	\bigl\{\eta(e_i )V'(\eta(e_i ))-u_i  V'(\eta(e_i ))-\eta(e_i )\nabla_i \sigma^\delta(u) \\
	& \phantom{\qquad {}-tN^{-d}\Bigl|\overline{D_N}\Bigr|\sum_{i =1}^d
	\int_{\mathcal{X}\times\real^d}}
	{}+u_i \nabla_i \sigma^\delta(u)-1
	\bigr\}\,p^{N,\delta}(d\eta\,du)+	tN^{-d}|\partial D_N|,
\end{align*}}%
where $p^{N,\delta}$ is the coupled measure introduced by \eqref{coupled}.
Applying the same argument as in the Section~6 of \cite{FS97}
with Propositions~\ref{prop2.8} and \ref{apriori-SDE}, we
conclude \eqref{eq4.4}.

\subsection{Derivation of the macroscopic equation in general cases}
Let us remove the assumption imposed at Section~\ref{HDL-smooth-init}
and complete the proof of Theorem~\ref{hydro}.
For this aim, we prepare the following lemma:
\begin{lemma}\label{approx-SDE}
	Let $\phi_t$ and $\tilde{\phi}_t$ be the solution of \eqref{SDE1}
	with common Gaussian processes $\{\tilde{w}_t(x);\,x\in D_N\}$ and let
	$h^N$ and $\tilde{h}^N$ be the macroscopic height
	variables corresponding to $\phi_t$ and $\tilde{\phi}_t$, respectively.
	Then, for every $t>0$ and $N\ge1$
	\[E\|h^N(t)-\tilde{h}^N(t)\|_{-1,N}^2
	\le E\|h^N(0)-\tilde{h}^N(0)\|_{-1,N}^2\]
	holds.
\end{lemma}
Noting that $\phi_t$ and $\tilde{\phi}_t$ satisfy
the same boundary condition, the quite same argument
as in the proof of Proposition~\ref{apriori-SDE}
can be applicable. We therefore omit the proof.

We shall approximate $h_0\in L^2(D)$ by $h_0^\epsilon\in C^2_0(D)$
in the sense of
\[\lim_{\epsilon\to0}\|h_0^\epsilon-h_0\|_{H^{1}(D)^*}=0.\]
Let $\phi^{\epsilon}_t$ be the solution of \eqref{SDE1}
with common Gaussian processes $\{\tilde{w}_t(x);\,x\in D_N\}$ as $\phi_t$
and with initial data $\phi_0^{\epsilon}$ which are defined by
\[\phi_0^{\epsilon}(x)=
\begin{cases}
N^{d+1}\int_{B(x/N,1/N)}h_0^\epsilon(\theta),& x\in D_N, \\
0,& x\in \integer^d\smallsetminus D_N.
\end{cases}\]
Letting $h^{N,\epsilon}$ be the macroscopic height
variable corresponding to $\phi^{\epsilon}_t$
and $h^{\epsilon}$ be the solution of \eqref{PDE1}
with the initial data $h_0^\epsilon$, we then have
\begin{align*}
	E\|&h^N(t)-h(t)\|^2_{H^{1}(D)^*} \\
	&\le 4E\|h^N(t)-h^{N,\epsilon}(t)\|^2_{H^{1}(D)^*}
	+4E\|h^{N,\epsilon}(t)-h^{\epsilon}(t)\|^2_{H^{1}(D)^*} \\
	& \qquad {}+8\|h^{\epsilon}(t)-h(t)\|^2_{H^{1}(D)^*} \\
	&\le 4E\|h^N(t)-h^{N,\epsilon}(t)\|^2_{-1,N}
	+4E\|h^{N,\epsilon}(t)-h^{\epsilon}(t)\|^2_{H^{1}(D)^*} \\
	& \qquad {}+8\|h^{\epsilon}(t)-h(t)\|^2_{H^{1}(D)^*} 
\end{align*}
by \eqref{sobolev-norms}. Here, applying the result in Section~\ref{HDL-smooth-init},
Lemma~\ref{approx-SDE} and Proposition~\ref{prop-PDE-approx2},
we complete the proof of Theorem~\ref{hydro}. 

\bibliographystyle{amsplain}
\bibliography{hydro}

\providecommand{\bysame}{\leavevmode\hbox to3em{\hrulefill}\thinspace}
\providecommand{\MR}{\relax\ifhmode\unskip\space\fi MR }
\providecommand{\MRhref}[2]{%
  \href{http://www.ams.org/mathscinet-getitem?mr=#1}{#2}
}
\providecommand{\href}[2]{#2}
\begin{thebibliography}{10}

\bibitem{B10}
V.~Barbu, \emph{Nonlinear differential equations of monotone types in banach
  spaces}, Springer, 2010.

\bibitem{BE91}
J.~F. Blowey and C.~M. Elliott, \emph{The {C}ahn-{H}illiard gradient theory for
  phase separation with nonsmooth free energy. {I}. {M}athematical analysis},
  European J. Appl. Math. \textbf{2} (1991), no.~3, 233--280. \MR{1123143
  (93a:35025)}

\bibitem{DGI00}
J.-D. Deuschel, G.~Giacomin, and D.~Ioffe, \emph{Large deviations and
  concentration properties for $\nabla\varphi$ interface models},
  Probab.~Theory Relat.~Fields \textbf{117} (2000), 49--111.

\bibitem{DNV11}
J.-D. Deuschel, T.~Nishikawa, and Y.~Vignaud, \emph{Hydrodynamic limit for the
  interface model with general potential}, in preparation, 2015.

\bibitem{E98}
Lawrence~C. Evans, \emph{Partial differential equations}, American Mathematical
  Society, 1998.

\bibitem{F82a}
J.~Fritz, \emph{Stationary measures of stochastic gradient systems, infinite
  lattice models}, Z. Wahrsch. Verw. Gebiete \textbf{59} (1982), no.~4,
  479--490. \MR{MR656511 (83j:60108)}

\bibitem{F05}
T.~Funaki, \emph{Stochastic interface models}, Lectures on probability theory
  and statistics, Lecture Notes in Math., vol. 1869, Springer, 2005,
  pp.~103--274.

\bibitem{FS97}
T.~Funaki and H.~Spohn, \emph{Motion by mean curvature from the
  {Ginzburg-Landau} $\nabla\phi$ interface model}, Commun.~Math.~Phys.
  \textbf{185} (1997), 1--36.

\bibitem{N02a}
T.~Nishikawa, \emph{Hydrodynamic limit for the {Ginzburg-Landau} $\nabla\phi$
  interface model with a conservation law}, J.~Math.~Sci.~Univ.~Tokyo
  \textbf{9} (2002), 481--519.

\bibitem{N03}
\bysame, \emph{Hydrodynamic limit for the {Ginzburg-Landau} $\nabla\phi$
  interface model with boundary conditions}, Probab.~Theory~Relat.~Fields
  \textbf{127} (2003), 205--227.

\end{thebibliography}
\vspace{5mm}

\end{document}